\documentclass[11pt, oneside]{amsart}
\usepackage{amsmath,ifthen, amsfonts, amssymb,
srcltx,
   amsopn, textcomp}

\usepackage{hyperref}
\usepackage{graphicx}
\usepackage{xypic}

\usepackage{marginnote}
\usepackage{overpic}

\usepackage[small]{caption}

\newtheorem{thm}{Theorem}[section]
\newtheorem{lem}[thm]{Lemma}

\newtheorem{cor}[thm]{Corollary}

\newtheorem{prop}[thm]{Proposition}

\newtheorem*{claim}{Claim}

\theoremstyle{definition}
\newtheorem{defn}[thm]{Definition}
\newtheorem{rem}[thm]{Remark}
\newtheorem{exmp}[thm]{Example}

\newcommand{\field}[1]{\mathbb{#1}}
\newcommand{\integers}{\ensuremath{\field{Z}}}

\newcommand{\wt}[1]{\widetilde{#1}}
\newcommand{\ola}[1]{\overleftarrow{ #1 }}
\newcommand{\ora}[1]{\overrightarrow{ #1 }}
\newcommand{\la}{\langle}
\newcommand{\ra}{\rangle}

\DeclareMathOperator{\Carrier}{Carrier}

\DeclareMathOperator{\lcm}{lcm}

\DeclareMathOperator{\stab}{Stab}

\begin{document}
\title{Classifying Finite Dimensional Cubulations of Tubular Groups}
\author{Daniel J. Woodhouse}

\begin{abstract}
  A tubular group is a group that acts on a tree with $\mathbb{Z}^2$ vertex stabilizers and $\mathbb{Z}$ edge stabilizers.
  This paper develops further a criterion of Wise and determines when a tubular group acts freely on a finite dimensional CAT(0) cube complex.
  As a consequence we offer a unified explanation of the failure of separability by revisiting the non-separable 3-manifold group of Burns, Karrass and Solitar and relating it to the work of Rubinstein and Wang.
  We also prove that if an immersed wall yields an infinite dimensional cubulation then the corresponding subgroup is quadratically distorted.
\end{abstract}

\maketitle

\section{Introduction} \label{introduction}

A \emph{tubular group} $G$ is a group which splits as a graph of groups with $\mathbb{Z}^2$ vertex groups and $\mathbb{Z}$ edge groups.
A tubular group is the fundamental group of a graph of spaces $X$ with each vertex space homeomorphic to a torus and each edge space homeomorphic to a cylinder.
The graph of spaces $X$ is a \emph{tubular space}.
In this paper all tubular groups will be finitely generated and thus tubular spaces will be compact.
Examples of tubular groups were used in \cite{BradyBridson00} by Brady and Bridson to provide groups with isoperimetric function $n^{\alpha}$ for all $\alpha$ in a dense subset of $[2, \infty)$.
Cashen provided a method of determining when two tubular groups are quasi-isometric \cite{Cashen10}.
Wise gave a criterion which determines whether or not a tubular group acts freely on a CAT(0) cube complex \cite{Wise13}, and characterised which tubular groups act cocompactly on a CAT(0) cube complex. 
This paper determines which tubular groups act on finite dimensional CAT(0) cube complexes.

 \begin{defn}
  An \emph{infinite cube} in a CAT(0) cube complex is the union of an ascending sequence of $n$-cubes $c_n$ of $\wt{X}$ such that $c_n$ is a subcube of $c_{n+1}$ for each $n$.
 \end{defn}

 \noindent Wise reduces the existence of cubulations to a combinatorial criterion called \emph{equitable sets}.
 Given an equitable set one can construct a finite set of \emph{immersed walls}.
  An immersed wall is a graph immersed $\pi_1$-injectively in $X$, such that $\wt{\Lambda}$ lifts to a 2-sided embedding $\wt{\Lambda} \rightarrow \wt{X}$.
 By $2$-sided we mean that the image of $\wt{\Lambda}$ in $\wt{X}$ is contained in a neighbourhood homeomorphic to $\wt{\Lambda} \times [-1,1]$. 
 The set of all such lifts give a $G$-invariant set $\mathcal{W}$ of \emph{walls}. The pair $(\wt{X}, \mathcal{W})$ is a \emph{wallspace}.
 In this paper, all references to ``immersed walls'' will be in reference to immersed walls obtained from an equitable set.
 Note that the theorems referring to immersed walls will not apply to any other kind of immersed walls.
 Section \ref{Contracting} defines the notion of a \emph{dilating wall} which can be recognized through a combinatorial criterion.
 As explained in Proposition \ref{prop:partition}, a wall not being dilated means its $G$-translates intersecting a vertex space in $\wt{X}$ can be partitioned into finitely many sets of pairwise nonintersecting walls. 

 The following is proven in Theorem \ref{theorem:MainA}

\begin{thm} \label{mainB}
 Let $X$ be tubular space, and $(\wt{X}, \mathcal{W})$ the wallspace obtained from a finite set of immersed walls in $X$. The following are equivalent: 
 \begin{enumerate}
  \item \label{infDimensional} The dual cube complex $C(\wt{X},\mathcal{W})$ is infinite dimensional.
  \item \label{infCube} The dual cube complex $C(\wt{X},\mathcal{W})$ contains an infinite cube.
  \item One of the immersed walls is dilated.
 \end{enumerate}
\end{thm}

\noindent Corollary~\ref{cor:decidable} states that it is decidable if a given immersed wall is dilated.
However, in contrast to cocompactness, there is no known simple criterion to determine whether or not a given tubular group acts on a finite dimensional CAT(0) cube complex.
See Example~\ref{exmp:DirectInfiniteDimensional} for an example of a direct proof that a specific tubular group does not posses an equitable set that can produce non-dilated immersed walls.

A group $G$ is \emph{separable} if every finitely generated subgroup $H \leqslant G$ is the intersection of all finite index subgroups containing $H$.
Burns, Karass, and Solitar gave the first example of a non-separable $3$-manifold group \cite{BurnsKarrassSolitar87}.
Niblo and Wise reproved this result \cite{NibloWiseNonEngulfing}.
Rubinstein and Wang produced an example of a non-embedded immersed surface $S$ in a graph manifold $M$ such that $\wt{S} \rightarrow \wt{M}$ is injective and any two $\pi_1M$ translates of $\wt{S}$ intersect \cite{RubinsteinWang98}. 
It follows that $\pi_1S$ is not separable in $\pi_1M$.
This is an application of the geometric interpretation for separability given by Scott \cite{Scott78}.
Theorem \ref{mainB} carries information about the separability of the associated codimension-1 subgroups of $G$, and provides a new proof of the non-subgroup separability of certain $3$-manifold groups that conceptually unifies the result of Burns, Karass, and Solitar with the geometric proof of Rubinstein and Wang (see Example \ref{example1}).

Finally, the following Theorem is a consequence of Theorem \ref{mainB} together with a study of the geometry of the walls.

\begin{thm} \label{mainD}
 Let the tubular space $X$ associated to the tubular group $G$ have a finite set of immersed walls.
 If the corresponding wallspace $(\wt{X}, \mathcal{W})$ has quasi-isometrically embedded walls, then the dual CAT(0) cube complex $C(\wt{X}, \mathcal{W})$ is finite dimensional.
\end{thm}

\noindent Example \ref{spiralExmp} demonstrates that quasi-isometrically embedded walls are not a necessary condition for finite dimensionality.

\subsection{Tubular Groups and their Cubulations} \label{subsect:tubgroups}

Let $G$ be a tubular group with associated tubular space $X$.
Let $X_v$ and $X_e$ denote vertex and edge spaces in this graph of spaces, where $v$ and $e$ are a vertex and edge in
the underlying graph $\Gamma$. Each $X_v$ is homeomorphic to the torus $S^1 \times S^1$ and each $X_e$ is homeomorphic to the cylinder $S^1 \times [-1,1]$.
Choosing an orientation for an edge $e$, let $\ola{e}, \ora{e}$ denote the initial and terminal vertices, let $\overleftarrow{X}_{e}, \overrightarrow{X}_{e}$ denote the corresponding boundary circles in $X_e$, let $\ola{f}_e : \ola{X}_e \rightarrow X_{\ola{e}}$ and let $\ora{f}_e : \ora{X}_e \rightarrow X_{\ora{e}}$ denote the attaching maps.
Let $\widetilde{X}_{\widetilde{v}}$ and $\widetilde{X}_{\widetilde{e}}$ denote vertex and edge spaces in the universal cover where $\widetilde{v}$ and $\widetilde{e}$ are a vertex and edge in the Bass-Serre tree $\widetilde{\Gamma}$.
There is a $G$ action on both $\wt{X}$ and $\wt{\Gamma}$ such that $g\wt{X}_{\wt{v}} = \wt{X}_{g\wt{v}}$.

Let $\alpha, \beta : S^1 \rightarrow S^1 \times S^1$ be a transverse pair of geodesic immersed circles in a torus.
The \emph{geometric intersection number} of $\alpha, \beta$ is the number of points $(s,t) \in S^1 \times S^1$ such that $\alpha(s) = \beta(t)$.
Let $[\alpha], [\beta]$ be a pair of homotopy classes of immersed circles in the torus.
The \emph{geometric intersection number} of $[\alpha], [\beta]$ is the geometric intersection number for any choice of transverse geodesic immersed circles representing these homotopy classes.
Let $\#[\alpha, \beta]$ denote the geometric intersection number.
If $B = \{\beta_i \}$ is a set of geodesic circles then let $\#[\alpha, B] := \sum_i \#[ \alpha, \beta_i ]$.
Viewing $[\alpha], [\beta]$ as elements of $\pi_1X_v \cong \mathbb{Z}^2$, we recall from \cite{Wise13} that $\#[\alpha, \beta] = \det \big[[\alpha], [\beta]\big]$.

An \emph{equitable set} for $G$ is a collection of sets, $\{S_v \}_{ v \in V(\Gamma)}$,
where each $S_v$ is a finite set of geodesic immersed circles, $\alpha: S^1 \rightarrow X_v$,
such that $\langle [\alpha] \mid \alpha \in S_v \rangle$ generates a finite index subgroup of $\pi_1X_v$ and
$\#[\overleftarrow{f}_{e}, S_{\overleftarrow{e}}] = \#[\overrightarrow{f}_{e}, S_{\overrightarrow{e}}]$ for all edges $e$ in $\Gamma$.
It is often convenient to allow $S_v$ to contain ``repeats'', i.e. distinct elements that are homotopic.
It is assumed, however, that distinct parallel circles in a vertex space have disjoint image.

The main theorem proven in \cite{Wise13} is

\begin{thm}[Wise]
 A tubular group $G$ acts freely on a CAT(0) cube complex if and only if there is an equitable set for $G$.
\end{thm}

\noindent As we repeatedly use this construction, here is an outline of how an equitable set provides a CAT(0) cube complex with a free $G$-action.

Let each $\alpha$ in the equitable set be given by a distinct copy of $S^1$ mapping into $X$.
Taking the disjoint union of these circles yields a map $\bigsqcup_{\alpha \in \bigcup S_v} S^1 \rightarrow X$.
As $\#[\overleftarrow{f}_{e}, S_{\overleftarrow{e}}] = \#[\overrightarrow{f}_{e}, S_{\overrightarrow{e}}]$, we can choose a bijection between the intersection points of $\overleftarrow{f}_{e}$ and the intersection points of $\overrightarrow{f}_e$.
Extend $\sqcup S^1$ to a space homeomorphic to a graph by attaching the endpoints of an arc to each pair of corresponding intersection points in our equitable set.
The connected components of the resulting graph are \emph{horizontal immersed walls}.
The immersions of the circles into $X$ can be extended along the arcs, by mapping each arc into the associated edge space.
For each edge space there is a \emph{vertical immersed wall}, which is a circle embedding along $S^1 \times \{0\} \subseteq S^1 \times [1,-1]$.
The complete collection of immersed walls obtained in this manner are the \emph{immersed walls obtained from the equitable set}.
All immersed walls mentioned in this paper will be obtained from the above construction.  
Note further, that an equitable set produces a finite set of immersed walls.

Let $\Lambda$ be an immersed wall.
Its universal cover $\widetilde{\Lambda}$ embeds in the universal cover $\widetilde{X}$.
%
%
\noindent A lift of $\wt{\Lambda}$ to $\wt{X}$ has an image contained in a product neighbourhood homeomorphic to $\wt{\Lambda} \times [-1, 1]$, so $\wt{X} - \wt{\Lambda}$ has exactly two connected components. 
The full set of such lifts and obtain a wallspace $(\widetilde{X}, \mathcal{W})$ which, via Sageev's construction \cite{Sageev95} gives a $G$-action on a CAT(0) cube complex (see \cite{HruskaWise13} for a comprehensive overview of Sageev's construction).

\begin{rem} \label{choices}
 In the above construction, the resulting dual CAT(0) cube complex depends on the choice of correspondence in each edge space, as well as how the interior of each arc embeds in the cylinder.
 Example \ref{arcChoiceDependent} shows finite dimensionality of the cubulation is dependent on the choice of correspondence between the intersection points at either end of the edge space. 
 By contrast, Theorem \ref{mainB} and the definition of a dilated immersed wall imply finite dimensionality is invariant of the subsequent choice of immersion of the interior of the arcs in the edge spaces.
 In \cite{Woodhouse214} we use the results presented here to show that acting freely on a finite dimensional CAT(0) cube complex is equivalent to virtual specialness for tubular groups
\end{rem}

\subsection{Acknowledgements}

I would like to thank Emily Stark who motivated me to find Example \ref{example1}; Dani Wise for support and advice; Mark Hagen for invaluable feedback; Michah Sageev, Eduardo Martinez-Pedroza, and Mathieu Carette for helpful conversations.

\section{A Motivating Example} \label{Motivation}

The next example motivates the definitions and theorems given later in this paper that identify which immersed walls give infinite cubulations.

\begin{exmp}\label{example1}
 Consider the tubular group $G = \langle a,b,t \mid [a,b] = 1, t a t^{-1} = b \rangle$.
 This is a cyclic HNN extension of $\mathbb{Z}^2 = \la a, b \ra.$
 The associated graph of spaces has a torus vertex space $X_v$, and a cylindrical edge space $X_e$ that is attached along geodesic paths associated to $a$ and $b$.
 Consider the immersed wall given by the equitable set $\{ a, ab^2 \}$ and the choice of arcs shown at right in Figure~\ref{fig:spaceX}.

 \begin{figure}
  \includegraphics[scale=0.2]{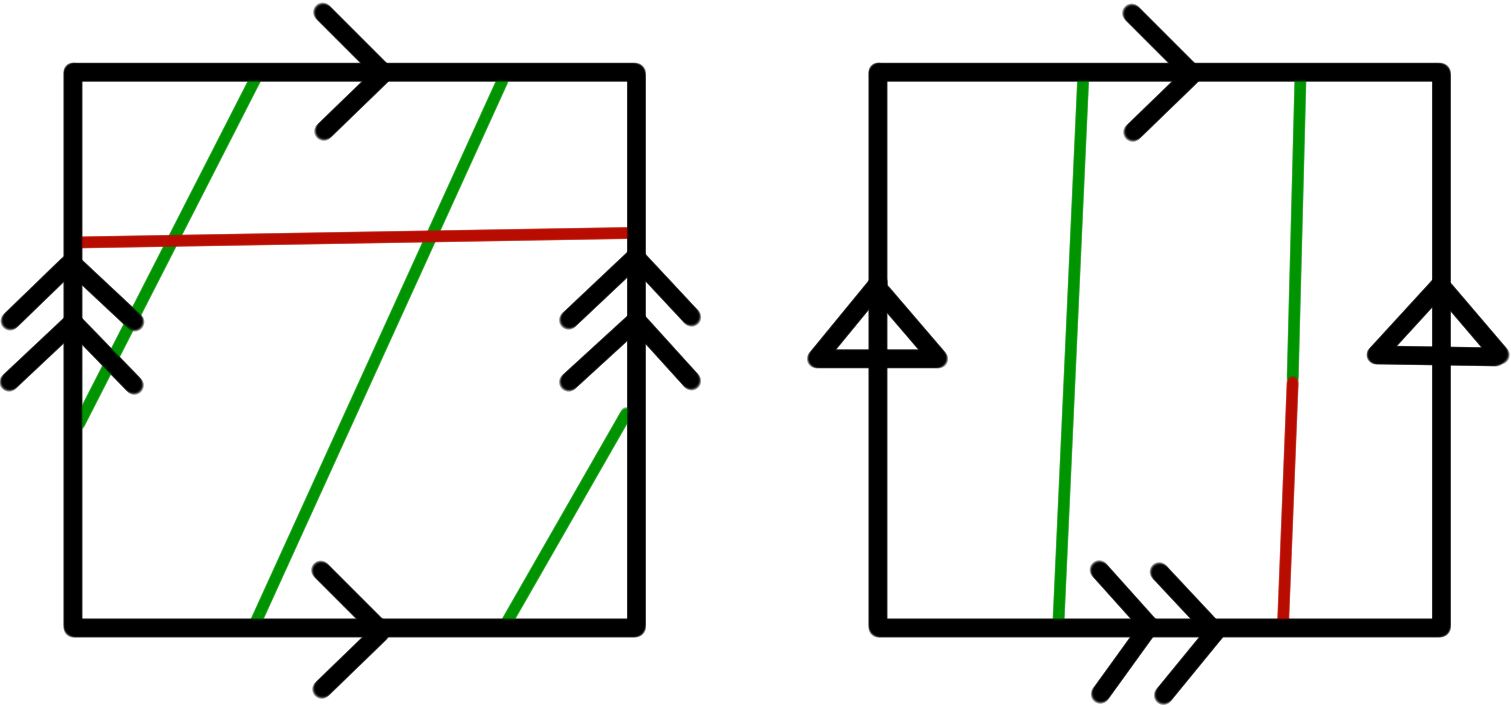}
 \caption{The graph of spaces $X$, with an immersed wall given by the equitable set.}
 \label{fig:spaceX}
 \end{figure}

 An infinite set of pairwise crossing walls in $(\wt{X}, \mathcal{W})$ will be exhibited
 to demonstrate that $C(\wt{X}, \mathcal{W})$ is infinite dimensional.
 The immersed wall in this example is the prototype for a general dilated immersed wall, which will be discussed in Section \ref{Contracting}.
 Choose a vertex space $\wt{X}_{\wt{v}_1}$ in the universal cover.
 The circle $ab^2$ in $X_v$ is covered by an infinite set of
 parallel lines in $\wt{X}_{\wt{v}_1}$.
 Index these lines consecutively by the elements of $\mathbb{Z}$.
 See the top of Figure \ref{fig:intersectionPattern}.

 Choose a vertex space $\wt{X}_{\wt{v}_2}$ adjacent to $\wt{X}_{\wt{v}_1}$ via an edge space attached
 along an $a$-line in $\wt{X}_{\wt{v}_1}$.
 The walls indexed $2n$ in $\wt{X}_{\wt{v_1}}$ travel
 through to $\wt{X}_{\wt{v}_2}$ where they either cover $ab^2$ or $a$.
 By adding $+1$ to each index if necessary, assume that they cover $ab^2$.
 Then the walls indexed $2n+1$ also travel through to $\wt{X}_{\wt{v}_2}$ where they cover $a$.
 Observe the wall indexed $2n$ in $\wt{X}_{\wt{v}_1}$
 can be indexed $n$ in $\wt{X}_{\wt{v}_2}$ and the indexing of the lines remains consecutive.
 This follows from how the immersed wall in the base space determines the walls in the universal cover
 (see Figure \ref{fig:intersectionPattern}).

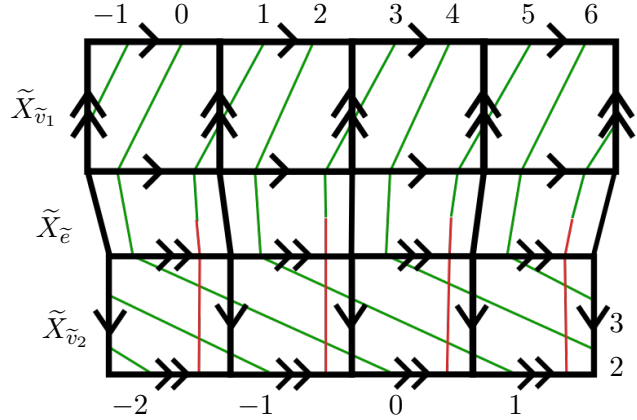
\begin{figure}
  \begin{overpic}[width=.6\textwidth,tics=10]{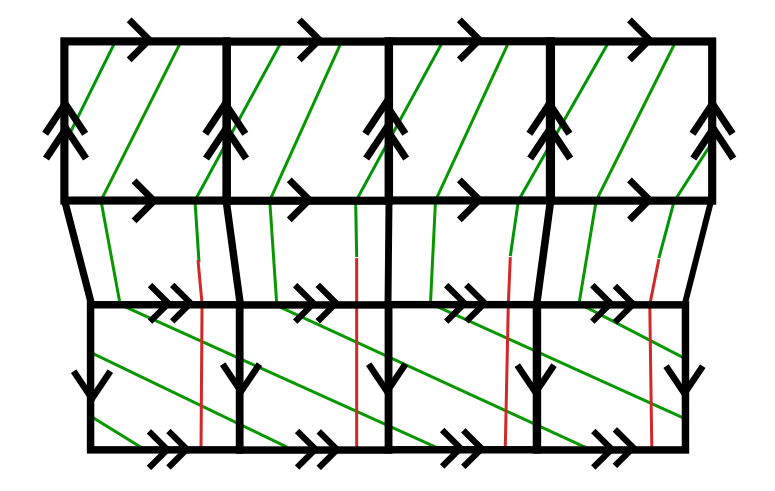}
     \put(9,61){$-1$}
     \put(22,61){$0$}
     \put(35,61){$1$}
     \put(44,61){$2$}
     \put(56,61){$3$}
     \put(65,61){$4$}
     \put(77,61){$5$}
     \put(87,61){$6$}
     \put(32,-1){$-1$}
     \put(56,-1){$0$}
     \put(75,-1){$1$}
     \put(12,-1){$-2$}
     \put(91,5){$2$}
     \put(91,12){$3$}
     \put(-4,46){$\wt{X}_{\wt{v}_1}$}
     \put(0,27){$\wt{X}_{\wt{e}}$}
     \put(1,11){$\wt{X}_{\wt{v}_2}$}
   \end{overpic}
  \caption{\label{fig:intersectionPattern}A portion of the universal cover $\wt{X}$, consisting of parts of $\wt{X}_{\wt{v}_1}$, $\wt{X}_{\wt{e}}$, and $\wt{X}_{\wt{v}_2}$. The numbers at the top show the indexing of the walls in $\wt{X}_{\wt{v}_1}$ and the numbers at the bottom show the indexing of the walls in $\wt{X}_{\wt{v}_2}$.}
\end{figure}
%
 The odd indexed walls in $\wt{X}_{\wt{v}_1}$ are now parallel to $a$ in $\wt{X}_{\wt{v}_2}$.
 Thus the even walls in $\wt{X}_{\wt{v}_1}$ cross the odd walls in $\wt{X}_{\wt{v}_1}$.

 We will now show that the walls indexed $1,2,4,8, \ldots, 2^n, \ldots$ in $\wt{X}_{\wt{v}_1}$ are pairwise intersecting.
 Indeed, as shown above, the wall numbered $1$ intersects all the even walls in
 $\wt{X}_{\wt{v}_1}$.
 To repeat the above argument, consecutively index the walls intersecting $\wt{X}_{\wt{v}_2}$ as $ab^2$-lines,  by the elements of $\mathbb{Z}$, such that the walls indexed $2,4,8, \ldots, 2^n, \ldots$ in $\wt{X}_{\wt{v}_1}$ are then indexed $1,2,4, \ldots ,2^{n-1}, \ldots$ in $\wt{X}_{\wt{v}_2}$.
 Arguing inductively, given the vertex space $\wt{X}_{\wt{v}_i}$ containing walls parallel to $ab^2$ indexed $1,2,4,8, \ldots, 2^n, \ldots$, there exists an adjacent vertex space $\wt{X}_{\wt{v}_{i+1}}$ where the wall indexed $1$ will cross all the other walls which can be re-indexed $1,2,4,8, \ldots, 2^n, \ldots$ in
 $\wt{X}_{\wt{v}_{i+1}}$.
 Therefore there is an infinite family of pairwise crossing walls intersecting $\wt{X}_{\wt{v_1}}$.
 It will be shown in Proposition \ref{WeakInfDimInfCube} that such a family determines an infinite cube.
\end{exmp}

\begin{rem}
 The existence of an infinite set of pairwise crossing walls in $\wt{X}$ implies the non-separability of the associated codimension-1 subgroup by Scott's criterion \cite{Scott78}.
 The tubular group of Example \ref{example1} is the $3$-manifold group that Burns, Karrass and Solitar proved was not subgroup separable in \cite{BurnsKarrassSolitar87}.
  The group that they used to prove non subgroup separability is a subgroup of the stabilizer of this wall.
  Niblo and Wise reproved this result in \cite{NibloWiseNonEngulfing} by finding a family of subgroups that are not contained in any finite index subgroup. The stabilizer of the wall belongs to that family.
\end{rem}

\section{Infinite Dimensional Cubulations Must Contain Infinite Cubes} \label{InfDimCubes}

 \noindent The goal of this section is to prove Proposition \ref{WeakInfDimInfCube}, a weaker version of implication (\ref{infDimensional})$\Rightarrow$(\ref{infCube}) from Theorem \ref{mainB}.

\begin{rem} \label{intersections}
 A distinction will be made between two types of intersections occurring between a pair of walls in $(\wt{X}, \mathcal{W})$ (see Figure \ref{fig:intersectionTypes}).
 \begin{description}
  \item[Regular] An intersection in a vertex space as between non-parallel lines in $\mathbb{R}^2$.
  \item[Non-Regular] An intersection in a vertex space as overlapping parallel lines in $\mathbb{R}^2$, or an intersection occurring in an edge space.
 \end{description}

 \noindent Note that walls may intersect more than once, and a given pair of walls may have multiple regular and non-regular intersections.
 Note that if $\wt{\Lambda}$ and $\wt{\Lambda}'$ regularly intersect in $\wt{X}_{\wt{v}}$, then $\wt{\Lambda}$ and $g\wt{\Lambda}'$ also regularly intersect in $\wt{X}_{\wt{v}}$ for all $g \in G_{\wt{v}}$.
 Two walls $\wt{\Lambda}_1, \wt{\Lambda}_2$ that intersect a vertex space $\wt{X}_{\wt{v}}$, but do not intersect regularly are \emph{parallel in $\wt{X}_{\wt{v}}$}.

 \begin{figure}
 \includegraphics[scale=0.3]{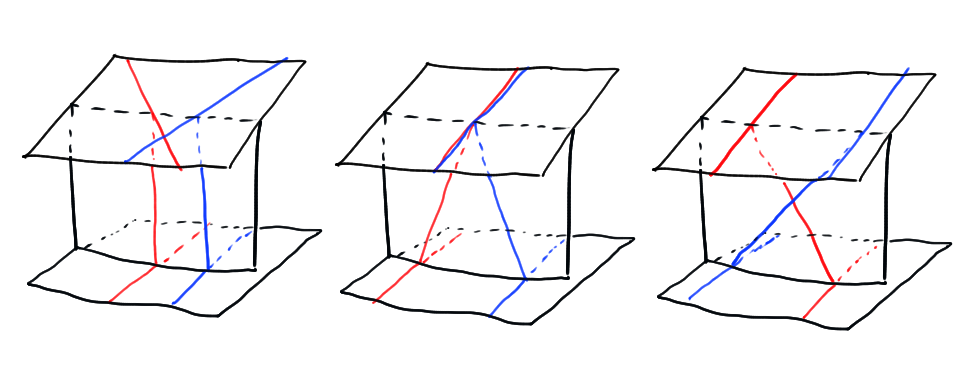}
 \caption{A regular intersection on the left and non-regular intersections on the right and center.}
 \label{fig:intersectionTypes}
\end{figure}

\end{rem}

 The following short lemma is used extensively throughout the paper.

 \begin{lem} \label{lemma:paralem1}
  Let $\wt{\Lambda}_1$ and $\wt{\Lambda}_2$ be walls in $(\widetilde{X}, \mathcal{W})$ that intersect $\wt{X}_{\wt{v}}$.
  Let $\widetilde{Y} \subset \widetilde{X}$ be a subtree of spaces containing $\wt{X}_{\wt{v}}$ such that $\wt{\Lambda}_1, \wt{\Lambda}_2$ have no regular intersections in $\wt{Y}$.
  Then every vertex or edge space in $\widetilde{Y}$ that intersects $\wt{\Lambda}_1$ also intersects $\wt{\Lambda}_2$.
 \end{lem}

\begin{proof}
Let $\wt{X}_{\wt{u}}$ be a vertex space in $\wt{Y}$.
If $\wt{\Lambda}_1 \cap \wt{X}_{\wt{u}}$ and $\wt{\Lambda}_2 \cap \wt{X}_{\wt{u}}$ are a pair of parallel (possibly overlapping) lines in $\wt{X}_{\wt{u}}$ then $\wt{\Lambda}_1, \wt{\Lambda}_2$ intersect an identical set of edge spaces adjacent to $\wt{X}_{\wt{u}}$, and thus an identical set of vertex spaces adjacent to $\wt{X}_{\wt{u}}$.
Therefore it can be deduced inductively, starting with $\wt{X}_{\wt{v}}$, that $\wt{\Lambda}_1$ and $\wt{\Lambda}_2$ must intersect an identical set of vertex and edge spaces in $\wt{Y}$.
\end{proof}

\begin{lem} \label{lemma:Helly}
 Let $S$ be a finite set of walls in $(\widetilde{X}, \mathcal{W})$ that pairwise intersect.
 There exists a vertex space $\widetilde{X}_{\widetilde{v}}$ that intersects every wall in $S$.
\end{lem}

\begin{proof}
 Each wall in $S$ maps to a connected subset of the Bass-Serre tree. These subsets pairwise intersect, therefore by
 application of Helly's Theorem for trees, there exists some vertex in the total intersection.
 Therefore every wall in $S$ intersects the corresponding vertex space.
\end{proof}

\noindent A set of horizontal walls in $( \widetilde{X}, \mathcal{W} )$ that pairwise regularly intersect is \emph{crossing}.

\begin{prop} \label{WeakInfDimInfCube}
 If there are crossing sets of arbitrary finite cardinality in $(\wt{X},\mathcal{W})$, then there exists an infinite cube in $C(\wt{X},\mathcal{W})$. 
\end{prop}

\begin{proof}[Proof of Proposition \ref{WeakInfDimInfCube}]
Let $\{K_i\}$ be a sequence of crossing sets such that $|K_i| \rightarrow \infty$ as $i \rightarrow \infty$.
There are only finitely many wall orbits so assume that each $K_i$ consists exclusively of translates of a single wall $\wt{\Lambda}$.
By Lemma \ref{lemma:Helly}, for every $i$ there exists a vertex $\widetilde{v}_i$ in $\widetilde{\Gamma}$  such that every wall in $K_i$ intersects $\wt{X}_{\wt{v}_i}$. 
Since $G$ acts cocompactly on $\wt{\Gamma}$, after passing to a subsequence of $\{ K_i \}$, it can be assumed that the $\{\wt{v_i}\}$ lie in a single $G$-orbit, $G\wt{v}$. 
Choose $g_i$ such that  $g_i\widetilde{v}_i = \widetilde{v}$ and replace $K_i$ with $g_i^{-1} K_i$.
Therefore all walls in $K_i$ intersect the vertex space $\widetilde{X}_{\widetilde{v}}$ for each $i$.

There are only finitely many parallelism classes of lines in $\widetilde{X}_{\widetilde{v}}$ belonging to walls in $\mathcal{W}$, 
 so there must exist $Q_0$, a set of walls parallel in $\widetilde{X}_{\widetilde{v}}$, with the following property:
\emph{there is no upper bound on the size of crossing subsets of $Q_0$}.
Given $Q_0$, there must exist a crossing set $\{ \wt{\Lambda}_1, \wt{\Lambda}_1' \} \subset Q_0$.
The regular intersection points between $\wt{\Lambda}_1, \wt{\Lambda}_1'$ must be in vertex spaces other than $\wt{X}_{\wt{v}}$.
Choose $\wt{\Lambda}_1, \wt{\Lambda}_1'$ such that they intersect in $\widetilde{X}_{\widetilde{v}_1}$,
and the distance from $\widetilde{v}$ to $\widetilde{v_1}$ in $\widetilde{\Gamma}$ is minimized  over all such choices of intersection points and choices of $\wt{\Lambda}_1, \wt{\Lambda}_1'$.

By Lemma \ref{lemma:paralem1}, setting $\wt{Y} \subset \wt{X}$ to be the union of vertex and edge spaces between and including $\wt{v}$ and $\wt{v}_1$, every other wall in $Q_0$ also intersects $\widetilde{X}_{\widetilde{v_1}}$. 
Repeating the above argument there exists a set $Q_1$, of walls parallel in $\widetilde{X}_{\widetilde{v_1}}$, such that there is no upper bound on the size of crossing subsets.
Either $\wt{\Lambda}_1$ or $\wt{\Lambda}_1'$ will therefore intersect every wall in $Q_1$, so assume that it is $\wt{\Lambda}_1$.
Iterate this argument to obtain a sequence of walls $\{ \wt{\Lambda}_i \}_{i=1}^{\infty}$ that pairwise intersect.

\begin{figure}
 \includegraphics[scale=0.15]{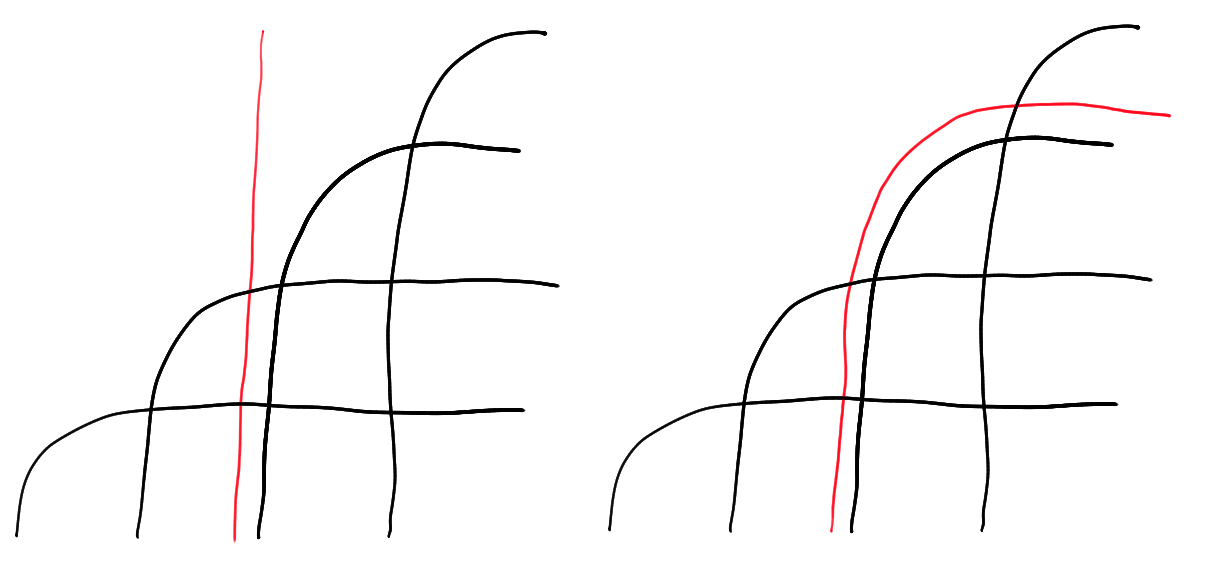}
 \caption{The potential obstruction to an infinite cube on the left, and the concluding situation on the right. Compare with the picture in \cite{HruskaWise13}.}
 \label{fig:InfCubeObstruction}
\end{figure}

Let $\{ \wt{\Lambda}_i \}_{i=1}^{\infty}$ be the (sub)sequence of walls re-indexed so that they are ordered consecutively in $\wt{X}_{\wt{v}}$. 
An ascending sequence of $n$-cubes $c_n$ in $C(\wt{X}, \mathcal{W})$ will now be constructed.
Orient all walls not intersecting $\wt{X}_{\wt{v}}$ towards $\wt{X}_{\wt{v}}$.
Let $c_0$ be the canonical $0$-cube given by orienting all walls towards some $x \in \wt{\Lambda}_1 \cap \wt{X}_{\wt{v}}$.
By induction, there is an ascending sequence of cubes $c_0 \subset \cdots \subset c_n$
dual to the walls $\wt{\Lambda}_1, \ldots, \wt{\Lambda}_{n+1}$.
Suppose that for some orientation of $\wt{\Lambda}_1, \ldots, \wt{\Lambda}_{n+1}$ the orientation of $\wt{\Lambda}_{n+2}$ could not be reversed (see Figure \ref{fig:InfCubeObstruction}).
This means that there exists a wall that is parallel to $\wt{\Lambda}_{n+2}$, that intersects $\wt{X}_{\wt{v}}$, and that lies between $x$ and  $\wt{\Lambda}_{n+2}$.
There can only exist finitely many such walls and, by Lemma \ref{lemma:paralem1}, setting $\wt{Y} = \wt{X}$ they all intersect $\wt{\Lambda}_i$ for $i \neq n+2$.
Replace $\wt{\Lambda}_{n+2}$ with the closest such wall to $x$ to obtain an $(n+1)$-cube $c_{n+1}$. 
\end{proof}

\section{Dilating Walls and Infinite Dimensional Cubulations} \label{Contracting}

 Let $\wt{X}_{\wt{v}}$ be a vertex space covering $X_v$.
 Identify $\wt{X}_{\wt{v}}$ with $\mathbb{R}^2$ so that a $G_{\wt{v}}$-orbit in $\wt{X}_{\wt{v}}$ is identified with $\mathbb{Z}^2 \subset \mathbb{R}^2$.
 Let $\wt{\Lambda}$ denote a wall covering the immersed wall $\Lambda$, intersecting $\wt{X}_{\wt{v}}$.
 The intersection $\wt{\Lambda} \cap \wt{X}_{\wt{v}}$ is a line with rational slope, stabilized by an infinite cyclic subgroup of $G_{\wt{v}}$. 
 Since the slope is rational there is a maximal infinite cyclic subgroup $H \leqslant G_{\wt{v}}$ that stabilizes a line perpendicular to $\wt{\Lambda} \cap \wt{X}_{\wt{v}}$. 
 The elements of $H$ are \emph{perpendicular to $\wt{\Lambda} \cap \wt{X}_{\wt{v}}$}.
  Note that there are finitely many $G_{\wt{v}}$-orbits of walls intersecting $\wt{X}_{\wt{v}}$.
 Given vertices $\wt{u}$, $\wt{v}$ in $\wt{\Gamma}$, the \emph{carrier of $\wt{u}, \wt{v}$}, denoted $\Carrier(\wt{u}, \wt{v})$, is the union of all vertex and edge spaces corresponding to the vertices and edges on the geodesic in the Bass-Serre tree $\wt{\Gamma}$ between and including $\wt{u}$ and $\wt{v}$. 

 The following is elementary:

 \begin{lem}\label{lemma:CyclicMatching}
  Let $T$ be a torus and $C \rightarrow T$ an immersed geodesic circle, with a choice of lift $\wt{C} \rightarrow \wt{T}$.
  Given $a, b \in \pi_1T - \stab(\wt{C})$,
  there exists $m_a, m_b \in \mathbb{Z} - \{0\}$ such that $a^{m_a} \wt{C} = b^{m_b} \wt{C}$. Moreover
  \[\dfrac{m_a}{m_b} = \dfrac{\#[C, b]}{\#[C, a]}.
   \]
 \end{lem}

 \begin{lem} \label{lemma:WallOrbits}
 Let $\wt{\Lambda}$ be a wall intersecting the vertex spaces $\wt{X}_{\wt{v}_0}$ and $\wt{X}_{\wt{v}_n}$, where $\wt{v}_1, \ldots , \wt{v}_{n-1}$ is the sequence of vertices on the geodesic between $\wt{v}_0$ and $\wt{v}_n$ in $\wt{\Gamma}$.
 Let $\wt{e}_i$ be the edge between $\wt{v}_{i-1}$ and $\wt{v}_i$.
 Let $G_{\wt{e}_i} = \la g_i \ra$ for $1 \leq i \leq n$.
 Let $\stab(\wt{\Lambda})\cap G_{\wt{v}_i} = \la \rho_i \ra$.
 Let $g_0 \in G_{\wt{v}_0}$ and $g_{n+1} \in G_{\wt{v}_n}$ be perpendicular to $\wt{\Lambda} \cap \wt{X}_{\wt{v}_0}$ and $\wt{\Lambda} \cap \wt{X}_{\wt{v}_n}$ respectively.
 Then there exist $\ola{m}$, $\ora{m} \in \mathbb{Z} - \{0\}$ such that $g_0^{\ola{m}}  \wt{\Lambda} = g_{n+1}^{\ora{m}} \wt{\Lambda}$ and
 \[
  \frac{\ola{m}}{\ora{m}} = \frac{\ola{m}_0 \cdots \ola{m}_n}{\ora{m}_0 \cdots \ora{m}_n}
 \]
 \noindent where $\ola{m}_i = \#[\rho_i, g_{i+1}]$ and $\ora{m}_i = \#[\rho_i, g_i]$ for $0 \leq i \leq n$.
 Moreover the walls in $\{ g_0^{\ola{m}r}\wt{\Lambda} \}_{r\in \mathbb{Z}}$ do not intersect within $\Carrier(\wt{v}_0, \wt{v}_n)$.
 \end{lem}

 \begin{proof}
  For each $0 \leq i \leq n$ apply Lemma \ref{lemma:CyclicMatching} to $g_i, g_{i+1} \in G_{\wt{v}_i}$ to deduce that $ g_i^{\ola{m}_i} \wt{\Lambda} = g_{i+1}^{\ora{m}_i} \wt{\Lambda}$.
  Let $\ola{m} = \ola{m}_0 \cdots \ola{m}_{n}$ and $\ora{m} = \ora{m}_0\cdots \ora{m}_{n}$.
  Observe that for $ 0 \leq i\leq n$ we have:
  \[
   g_0^{\ola{m}} \wt{\Lambda} = g_0^{\ola{m}_0\ola{m}_1 \cdots \ola{m}_{n}} \wt{\Lambda} = g_i^{\ora{m}_0 \cdots \ora{m}_i \ola{m}_{i+1} \cdots \ola{m}_{n}} \wt{\Lambda},
  \]
  \noindent so there are no intersections of walls in $\{ g_0^{\ola{m}r} \wt{\Lambda} \}_{r\in \mathbb{Z}}$ within either $\wt{X}_{\wt{v}_i}$ or $\wt{X}_{\wt{e}_i}$. 
 \end{proof}

  \begin{defn}
  Let $\wt{\Lambda}$ be a wall covering $\Lambda$ and let $\wt{X}_{\wt{v}}$ intersect $\wt{\Lambda}$.
  Observe that $\wt{\Lambda}$ also intersects $\wt{X}_{g\wt{v}}$ for any $g \in \stab(\wt{\Lambda})$.
  If $h\in G_{\wt{v}}$ is perpendicular to $\wt{\Lambda} \cap \wt{X}_{\wt{v}}$, then $ghg^{-1}$ is a perpendicular to $\wt{\Lambda} \cap \wt{X}_{g\wt{v}}$.
  From Lemma \ref{lemma:WallOrbits} there exist $\ola{m}, \ora{m} \in \mathbb{Z} - \{0\}$ such that $h^{\ola{m}} \wt{\Lambda} = gh^{\ora{m}}g^{-1} \wt{\Lambda}$ and $\{ h^{\ola{m}r} \wt{\Lambda} \}_{r\in \mathbb{Z}}$ contains no pair intersecting in $\Carrier(\wt{v},g\wt{v})$.
  The pair $\ola{m}, \ora{m}$ are \emph{$g$-shift exponents}.
  Note that $g$-shift exponents are not unique.
  If $\ola{m}, \ora{m}$ are $g$-shift exponents, then so are $\ola{m}n, \ora{m}n$ for $n \in \mathbb{Z} - \{0\}$.
  The \emph{dilation function of $\wt{\Lambda}$} is the map $R: \stab(\wt{\Lambda}) \rightarrow \mathbb{Q}^*$ with $g \mapsto \dfrac{\ola{m}}{\ora{m}}$, where $\ola{m}, \ora{m}$ are $g$-shift exponents.
  \end{defn}

  \begin{lem} \label{lemma:homomorphismR}
   The map $R : \stab(\wt{\Lambda}) \rightarrow \mathbb{Q}^*$ is a homomorphism, and does not depend on the choice of vertex space $\wt{X}_{\wt{v}}$ or perpendicular element $h \in G_{\wt{v}}$.
   Moreover, if $R'$ is the dilation function of $\wt{\Lambda}'$, where $\wt{\Lambda}' = g' \wt{\Lambda}$, then $R(g'^{-1}gg') = R'(g)$ for $g \in \stab(\wt{\Lambda}')$. 
  \end{lem}

  \begin{proof}
   Fix a choice of $\wt{X}_{\wt{v}}$ and let $h \in G_{\wt{v}}$ be a primitive element perpendicular to $\wt{\Lambda} \cap \wt{X}_{\wt{v}}$.
   To see that $R$ is well defined take any two choices of $g$-shift exponents $\ola{m}, \ora{m}$, and $\ola{m}', \ora{m}'$.
   Since $ h^{\ola{m}}  \wt{\Lambda} = g h^{\ora{m}} g^{-1} \wt{\Lambda}$ and $ h^{\ola{m}'} \wt{\Lambda} = g h^{\ora{m}'} g^{-1} \wt{\Lambda} \;$ we have $g h^{\ola{m}\ora{m}'}g^{-1} \wt{\Lambda} = h^{\ola{m}\ola{m}'} \wt{\Lambda} =  gh^{\ora{m}\ola{m}'} g^{-1} \wt{\Lambda}$ which implies that $\dfrac{\ola{m}}{\ora{m}} = \dfrac{\ola{m}'}{\ora{m}'}$.
   Now suppose that $h' \in G_{\wt{v}}$ is any other choice element perpendicular to $\wt{\Lambda} \cap \wt{X}_{\wt{v}}$, then $h' = h^n$ with $n\in\mathbb{Z} -\{0\}$.
   Therefore, $(h')^{\ola{m}} \wt{\Lambda} = h^{n\ola{m}}  \wt{\Lambda} = g h^{n\ora{m}} g^{-1} \wt{\Lambda} = g(h')^{\ora{m}}g^{-1} \wt{\Lambda}$, and $\frac{\ola{m}}{\ora{m}} = \frac{n\ola{m}}{n\ora{m}}$, so $\ola{m}, \ora{m}$ are also $g$-shift exponents with respect to $h'$. 

   To see that $R$ is independent of the choice of $\wt{X}_{\wt{v}}$ suppose $\wt{\Lambda}$ also intersects $\wt{X}_{\wt{w}}$.
   Let $h_{\wt{u}} \in G_{\wt{u}}$ be perpendicular to $\wt{\Lambda} \cap \wt{X}_{\wt{u}}$, and $h_{\wt{w}} \in G_{\wt{w}}$ be perpendicular to $\wt{\Lambda} \cap \wt{X}_{\wt{w}}$.
   Then by Lemma \ref{lemma:WallOrbits} there exist $p, q \in \mathbb{Z} - \{0\}$ such that $ h_{\wt{w}}^p  \wt{\Lambda} = h_{\wt{u}}^q \wt{\Lambda}$.
   Let $g \in \stab(\wt{\Lambda})$ and $\ola{m}, \ora{m}$ be $g$-shift exponents with respect to $\wt{X}_{\wt{u}}$ and deduce that $h_{\wt{w}}^{\ola{m}p} \wt{\Lambda} = h_{\wt{u}}^{\ola{m}q} \wt{\Lambda} = g h_{\wt{u}}^{\ora{m}q} g^{-1} \wt{\Lambda} =  g h_{\wt{w}}^{\ora{m}p} g^{-1} \wt{\Lambda}$.
   This means that $\ola{m}p, \ora{m}p$ are $g$-shift exponents with respect to $\wt{X}_{\wt{w}}$

   Let $\wt{\Lambda}' = g'\wt{\Lambda}$ with $g'\in G$.
   Let $g \in \stab( g'\wt{\Lambda} )$ then we know that $g'^{-1}gg' \in \stab(\wt{\Lambda})$ so there exist $g'^{-1}gg'$-shift exponents $\ola{m}, \ora{m}$ for $\wt{\Lambda}$ with respect to $\wt{X}_{\wt{v}}$ so that
   $h^{\ola{m}} \wt{\Lambda} = (g'^{-1}gg') h^{\ora{m}} (g'^{-1}gg')^{-1} \wt{\Lambda}$, where $h$ is perpendicular to $\wt{\Lambda}$ in $\wt{X}_{\wt{v}}$.
    This can be rewritten as $(g' h g'^{-1})^{\ola{m}} (g' \wt{\Lambda}) = g (g' h g'^{-1})^{\ora{m}} g^{-1} (g'\wt{\Lambda})$, where $g' h g'^{-1}$ is perpendicular to $\wt{\Lambda}'$ in $\wt{X}_{g'\wt{v}}$.
    This means that $\ola{m}, \ora{m}$ are $g$-shift exponents for $g'\wt{\Lambda}$ with respect to $\wt{X}_{g'\wt{v}}$, and $R'(g) = R(g'^{-1}gg')$.

    To see that $R$ is a homomorphism, let $g_1, g_2 \in \stab(\wt{\Lambda})$, let $\ola{m}_1, \ora{m}_1$ be $g_1$-shift exponents of $\wt{\Lambda}$ with respect to $\wt{X}_{\wt{v}}$ so
  $ h^{\ola{m}_1} \wt{\Lambda} = g_1 h^{\ora{m}_1} g_1^{-1} \wt{\Lambda}$, and let $\ola{m}_2, \ora{m}_2$ be $g_2$-shift exponents of $\wt{\Lambda}$ with respect to $\wt{X}_{g_1\wt{v}}$ so
  $ g_1h^{\ola{m}_2} g_1^{-1} \wt{\Lambda} =  g_2 g_1 h^{\ora{m}_2} g_1^{-1} g_2^{-1} \wt{\Lambda}$.
  Then $ h^{\ola{m}_1 \ola{m}_2} \wt{\Lambda} =  g_2 g_1 h^{\ora{m}_1\ora{m}_2} g_1^{-1} g_2^{-1}  \wt{\Lambda}$.
  Therefore $R(g_1g_2) = \dfrac{\ola{m}_1 \ola{m_2}}{\ora{m}_1 \ora{m_2}} = R(g_1)R(g_2)$.
  \end{proof}

  \begin{defn} \label{defn:dilated}
   If the dilation function $R$ has infinite image then $\widetilde{\Lambda}$ is \emph{dilated}.
   Otherwise $\widetilde{\Lambda}$ is \emph{non-dilated} (and the image of $R$ is either trivial or $\{-1,1\}$).
   An element $g \in \stab{\wt{\Lambda}}$ is \emph{dilated} if $R(g) \neq \pm1$.
   Lemma~\ref{lemma:homomorphismR} implies that if $\wt{\Lambda}$ is [non-]dilated, then all its $G$-translates are [non-]dilated.
   Therefore, it makes sense to say that an immersed wall $\Lambda$ is [non-]dilated if its associated walls are [non-]dilated walls.
  \end{defn}

  A wall $\wt{\Lambda}$ being dilated means there is an infinite family of walls, $\{ h^{r} \wt{\Lambda} \}_{r \in \mathbb{Z}}$, intersecting $\wt{X}_{\wt{v}}$, that become closer together while traveling from $\wt{X}_{\wt{v}}$ to $\wt{X}_{g\wt{v}}$.
  Considering higher powers of $g$, some infinite subset of these walls must intersect each other.
  The proof of the following proposition makes this idea precise.

 \begin{prop} \label{prop:dilatedInfinite}
  If $(\wt{X}, \mathcal{W})$ contains a dilated wall then there are crossing sets of arbitrary finite cardinality.
 \end{prop}

 \begin{proof}
 Let $\wt{\Lambda}_0$ be a dilated wall intersecting $\wt{X}_{\wt{v}}$ and let $h \in G_{\wt{v}}$ be perpendicular to $\wt{\Lambda}_0 \cap \wt{X}_{\wt{v}}$.
 Choose $g \in \stab(\wt{\Lambda}_0)$ such that $R(g) > 1$.
 By Lemma \ref{lemma:partition} the set of all $G$-translates of $\wt{\Lambda}_0$ intersecting $\wt{X}_{\wt{v}}$ that do not regularly intersect $\wt{\Lambda}_0$ in $\Carrier(\wt{v}, g\wt{v})$ can be partitioned:

  \[
   \mathcal{P} = \Big\{ \{ h^{\ola{n}r} \wt{\Lambda}_0 \}_{r \in \mathbb{Z}},  \ldots, \{ h^{\ola{n}r} \wt{\Lambda}_p \}_{r\in \mathbb{Z}}  \Big\},
  \]

  \noindent where $h^{\ola{n}} \wt{\Lambda}_i = g h^{\ora{n}}g^{-1}\wt{\Lambda}_i$  and $\dfrac{\ola{n}}{ \ora{n}} = R(g)$.

  Let $s \geq 1$ and $\wt{M} = h^{\ola{n}^s}\wt{\Lambda}_0$.
  We will show that $Q^{\infty}_s = \{ \wt{M}, g^{-1}\wt{M}, g^{-2}\wt{M}, \ldots, \}$ contains a cardinality $s$ crossing set.
  All $g^{-i}\wt{M}$ are distinct because if $g^{-p}\wt{M} = g^{-q}\wt{M}$ for $p\neq q$ then
  $
   h^{\ola{n}^s} \wt{\Lambda}_0 = g^{p-q} h^{\ola{n^s}} \wt{\Lambda}_0 = g^{p-q} h^{\ola{n^s}} g^{q-p} \wt{\Lambda}_0
  $
  which would imply $R(g^{p-q}) = 1$, but $R$ is a homomorphism so this would contradict $R(g) >1$.
  This implies that $Q^{\infty}_s = g(Q^{\infty}_s - \{\wt{M}\})$.
  Note that none of the walls $\wt{M}, \ldots, g^{1-s} \wt{M}$ regularly intersect $\wt{\Lambda}_0$ in $\Carrier(\wt{v},g\wt{v})$, since
  \[
   g^{-i}\wt{M} = g^{-i} h^{\ola{n}^s} \wt{\Lambda}_0  = g^{-i} h^{\ola{n}^i\ola{n}^{s-i}} g^i \wt{\Lambda}_0 = h^{\ora{n}^i \ola{n}^{s-i}}\wt{\Lambda}_0,
  \]
  \noindent and the right hand side belongs to the set $\{h^{\ola{n}r} \wt{\Lambda}_0\}_{r\in\mathbb{Z}} \in \mathcal{P}$ for $0 \leq i \leq s-1$.

  Suppose there exists $k$ such that $g^{-k}\wt{M}$ regularly intersects $\wt{\Lambda}_0$ in $\Carrier(\wt{v}, g\wt{v})$.
  Let $k$ be the minimal such value and note that $k \geq s$.
  Lemma \ref{lemma:paralem1} implies $g^{-i}\wt{M}$ does not regularly intersect $\wt{\Lambda}_0$ in $\Carrier(\wt{v},g\wt{v})$
  which implies that $g^{-k}\wt{M}$ regularly intersects $g^{-i}\wt{M}$ for $0 \leq i < k$.
  Then $Q^{k}_s = \{\wt{M}, \ldots , g^{-k}\wt{M}\}$ is a crossing set of cardinality at least $s+1$, since $g^{-k}\wt{M}$ regularly intersects all the other elements of $Q^k_s$ and $g^{-1}(Q^k_s - \{g^{-k}\wt{M}\}) = Q^k_s - \{\wt{M}\}$.

  Suppose that $g^{-i}\wt{M}$ does not regularly intersect $\wt{\Lambda}_0$ in $\Carrier(\wt{v}, g\wt{v})$ for all $i \geq 0$.
  Therefore each $g^{-i}\wt{M}$ belongs to some subset in the partition $\mathcal{P}$.
  Therefore there exists $r_i \in \mathbb{Z}$ and $\sigma(i) \in \{0,\ldots,p\}$ such that $g^{-i}\wt{M} = h^{\ola{n}r_i} \wt{\Lambda}_{\sigma(i)}$.
  We will obtain a contradiction by showing that $\{r_i \}$ is a bounded sequence, contradicting that $Q^{\infty}_s$ is an infinite set.
  From construction of the sequence

 \[
  g^{-1}h^{\ola{n}r_i} \wt{\Lambda}_{\sigma(i)} = g^{-1}(g^{-i} \wt{M}) = g^{-(i+1)}\wt{M} = h^{\ola{n}r_{i+1}} \wt{\Lambda}_{\sigma(i+1)},
 \]

 \noindent and from the properties of $\mathcal{P}$ it can be inferred that

 \[h^{\ola{n}r_{i+1}} \wt{\Lambda}_{\sigma(i+1)}
                                   = g^{-1}h^{\ola{n}r_i}\wt{\Lambda}_{\sigma(i)}
                                = h^{\ora{n}r_i} g^{-1} \wt{\Lambda}_{\sigma(i)}.\]
 \noindent Hence

  \[
  h^{\ola{n} r_{i+1} - \ora{n} r_i} \wt{\Lambda}_{\sigma(i+1)} = g^{-1}\wt{\Lambda}_{\sigma(i)},
 \]

 \noindent so $a_i = \ola{n} r_{i+1} - \ora{n} r_i$ is a bounded integer sequence because the right hand side can only be one of a finite number of walls.
 Therefore $r_{i+1} = \dfrac{\ora{n}}{\ola{n}}r_i + a_{i}$ is bounded since $ \dfrac{1}{R(g)} =\dfrac{\ora{n}}{\ola{n}} <1$.
 \end{proof}

 \noindent The following Lemma is used in the proof of Proposition \ref{prop:dilatedInfinite}.

 \begin{lem} \label{lemma:partition} Fix $g \in \stab(\wt{\Lambda}_0)$, and $h \in G_{\wt{v}}$ perpendicular to $\wt{\Lambda}_0 \cap X_{\wt{v}}$.
 Let $U$ be the set of $G$-translates of $\wt{\Lambda}_0$ that intersect $\wt{X}_{\wt{v}}$ and do not regularly intersect $\wt{\Lambda}_0$ in $\Carrier(\wt{v}, g\wt{v})$.
 Then there exist $\ola{n}, \ora{n} \in \mathbb{Z} - \{0\}$ such that $R(g) = \dfrac{\ola{n}}{\ora{n}}$, and a subset $\{ \wt{\Lambda}_1, \ldots, \wt{\Lambda}_p \} \subset U$  such that there is a partition of $P$,
 \[
   \mathcal{P} = \Big\{ \{ h^{\ola{n}r} \wt{\Lambda}_0 \}_{r \in \mathbb{Z}},  \ldots, \{ h^{\ola{n}r} \wt{\Lambda}_p \}_{r\in \mathbb{Z}}  \Big\},
  \]
 \noindent and such that $h^{\ola{n}} \wt{\Lambda}_i = g h^{\ora{n}} g^{-1} \wt{\Lambda}_i$ for each $i$.
 \end{lem}

 \begin{proof}
 Let $\Carrier(\wt{v}, g\wt{v})$ be the union of the consecutive vertex and edge spaces $\wt{X}_{\wt{v}_0}, \wt{X}_{\wt{e}_1}, \ldots , \wt{X}_{\wt{e}_{\ell}}, \wt{X}_{\wt{v}_{\ell}}$.
 For each $i$ there exists primitive $\rho_i \in G_{\wt{v}_i}$ such that $\stab(\wt{\Lambda}_0) \cap G_{\wt{v}_i} = \la \rho_i^{k_i} \ra $ for some $k_i \in \mathbb{Z} - \{0\}$.
  Let $\wt{\Lambda}$ be any wall in $U$.
  Since $\wt{\Lambda}$ doesn't regularly intersect $\wt{\Lambda}_0$ in $\wt{X}_{\wt{v}_i}$ observe that $\stab(\wt{\Lambda}) \cap G_{\wt{v}_i} = \la \rho_i^{j_i} \ra$ for some $j_i >0$. 

   Let $G_{\wt{e}_i} = \la g_i \ra$ for $0<i <\ell+1$, and let $g_0 = h$ and $g_{\ell+1} = g h g^{-1}$.
   By Lemma \ref{lemma:WallOrbits} there exists $\ola{m}, \ora{m} \in \mathbb{Z} - \{0\}$ such that $ h^{\ola{m}}  \wt{\Lambda} =  g h^{\ora{m}} g^{-1}  \wt{\Lambda}$ and
   \[
     \dfrac{\ola{m}}{\ora{m}}  =
                \dfrac{\ola{m}_0 \cdots \ola{m}_{\ell}}{\ora{m}_0 \cdots \ora{m}_{\ell}}
                             =  \dfrac{\#[\rho_0, g_1] \cdots \#[\rho_{\ell}, g_{\ell+1}]}{\#[\rho_0, g_0] \cdots \#[\rho_{\ell}, g_{\ell} ]},
   \]
  \noindent where the $j_i$-s have canceled on the right hand side.
  Hence $\dfrac{\ola{m}}{\ora{m}} = R(g)$ because the final ratio is independent of the values of $j_i$.
  There are only finitely many values that $j_i$ can take, so we can obtain $\ola{n}, \ora{n}$ such that for any $\wt{\Lambda}$, we have $h^{\ola{n}}  \wt{\Lambda} =  g h^{\ora{n}} g^{-1}  \wt{\Lambda}$.
  Since there are only finitely many $G_{\wt{v}}$-orbits of walls intersecting $\wt{X}_{\wt{v}}$ the partition is obtained by taking a finite number of representatives $\wt{\Lambda}_1, \ldots, \wt{\Lambda}_p$, such that the sets $\{ h^{\ola{n}r} \wt{\Lambda}_i \}_{r \in \mathbb{Z}}$ are a maximal collection of disjoint sets.
 \end{proof}

 We now consider the non-dilated walls.

 \begin{lem} \label{lemma:disjointWallOrbit}
  Let $\wt{\Lambda}$ be a non-dilated wall.
  For any $\wt{X}_{\wt{v}}$ intersected by $\wt{\Lambda}$, there exists $h \in G_{\wt{v}}$ perpendicular to $\wt{X}_{\wt{v}} \cap \wt{\Lambda}$ such that if $g \in \stab(\wt{\Lambda})$ then $ghg^{-1}\wt{\Lambda} = h\wt{\Lambda}$, and the walls in $\{h^r \wt{\Lambda} \}_{r\in \mathbb{Z}}$ are pairwise disjoint.
 \end{lem}

 \begin{proof}
 Choose a generating set $\{f_1,\ldots,f_n\}$ for $\stab(\wt{\Lambda})$, and fix $h_0 \in G_{\wt{v}}$ perpendicular to $\wt{X}_{\wt{v}} \cap \wt{\Lambda}$.
 By Lemma \ref{lemma:WallOrbits} and the assumption that $\wt{\Lambda}$ is nondilated, each $f_i$ has an $s_i$ such that $ h_0^{s_i}  \wt{\Lambda} =  f_i h_0^{\pm s_i}f_i^{-1}  \wt{\Lambda}$ and the walls $\{ h_0^{s_ir}  \wt{\Lambda} \}$ pairwise do not intersect in $\Carrier(\wt{v}, f_i\wt{v})$.
 Let $s = \lcm{s_i}$ and let $h = h_0^s$, then $h\wt{\Lambda} = gh^{\pm 1}g^{-1} \wt{\Lambda}$ and $\{h^r \wt{\Lambda}\}$ pairwise do not intersect in $\Carrier(\wt{v}, g\wt{v})$, for all $g\in \stab(\wt{\Lambda})$, as any $g$ can be expressed as the product of the generating set.

 For any other $\wt{X}_{\wt{w}}$ intersected by $\wt{\Lambda}$, there is a geodesic $\gamma: I \rightarrow \wt{\Lambda}$ with initial point $\wt{a} \in \wt{\Lambda} \cap \wt{X}_{\wt{v}}$ and endpoint in $\wt{\Lambda} \cap \wt{X}_{\wt{w}}$.
 Extend $\gamma$ to a geodesic $\gamma'$ such that the endpoint of $\gamma'$ is $g\wt{a}$ for some $g \in \stab(\wt{\Lambda})$.
 Therefore $h  \wt{\Lambda} =  g h^{\pm 1} g^{-1} \wt{\Lambda}$ and for all $r$, that $h^r \wt{\Lambda} \cap \wt{\Lambda} \cap \wt{X}_{\wt{w}} = \emptyset$ as $\wt{X}_{\wt{w}}$ lies in $\Carrier(\wt{v},g\wt{v})$.
 This implies that the walls in $\{ h^r \wt{\Lambda} \}_{r\in\mathbb{Z}}$ are pairwise disjoint.
 \end{proof}

 \begin{prop} \label{prop:nondilatedFinite}
 Let $(\wt{X}, \mathcal{W})$ be the wallspace obtained from a finite set of non-dilated immersed walls, then $C(\wt{X}, \mathcal{W})$ is finite dimensional.
 \end{prop}

 \begin{proof}
  Let $\{K_i \}$ be a sequence of finite sets of pairwise intersecting walls such that $\lim|K_i| = \infty$.
  By Lemma \ref{lemma:Helly} there exist vertex spaces $\wt{X}_{\wt{v}_i}$ such that each wall in $K_i$ intersects $\wt{X}_{\wt{v}_i}$. Since there are finitely many $G$-orbits of vertices in the Bass-Serre tree, choose a subsequence of $\{K_i\}$ and find a sequence $\{g_i\} \subset G$ such that all $\{ g_iK_i \}$ intersect a fixed $\wt{X}_{\wt{v}}$.
  Since there are finitely many $G_{\wt{v}}$-orbits of walls intersecting $\wt{X}_{\wt{v}}$, we may restrict to subsets, still of unbounded cardinality, such that all the walls lie in the same $G_{\wt{v}}$-orbit.

  Fix some $\wt{\Lambda}$ in $g_1K_1$.
  By Lemma \ref{lemma:disjointWallOrbit} there exists $h \in G_{\wt{v}}$ such that $\{  h^r \wt{\Lambda}\}_{r\in \mathbb{Z}} $ consists of pairwise non-intersecting walls.
  There are only finitely many $G_{\wt{v}}$-orbits of $\{ h^r \wt{\Lambda}\}_{r\in \mathbb{Z}}$, so by the pigeonhole principle there must be some $g_iK_i$ that has more than one wall in one of these orbits.
  This contradicts that the walls in $g_iK_i$ pairwise intersect.
 \end{proof}

 \begin{thm}  \label{theorem:MainA}
  Let $X$ be tubular space, and $(\wt{X}, \mathcal{W})$ the wallspace obtained from a finite set of immersed walls in $X$. The following are equivalent:
 \begin{enumerate}
  \item \label{C1}  $C(\wt{X},\mathcal{W})$ is infinite dimensional.
  \item \label{C2}  $C(\wt{X},\mathcal{W})$ contains an infinite cube.
  \item \label{C3} Some immersed wall is dilated.
 \end{enumerate}
 \end{thm}

 \begin{proof}
  (\ref{C3}) $\Rightarrow$ (\ref{C2}) follows from Proposition \ref{prop:dilatedInfinite} and Proposition \ref{WeakInfDimInfCube}.

  (\ref{C2}) $\Rightarrow$ (\ref{C1}) is immediate from the definition of an infinite cube.

  (\ref{C1}) $\Rightarrow$ (\ref{C3}) is the contrapositive of \ref{prop:nondilatedFinite}.
 \end{proof}

 \begin{cor} \label{cor:decidable}
  Let $X$ be tubular space, and $(\wt{X}, \mathcal{W})$ the wallspace obtained from a set of immersed walls in $X$.
  It is decidable whether or not $C(\wt{X}, \mathcal{W})$ is finite dimensional.
 \end{cor}

 \begin{proof}
  For each immersed wall $\Lambda$, the corresponding dilation function can be computed by using Lemmas \ref{lemma:CyclicMatching} and \ref{lemma:WallOrbits} to find their values on a finite generating set of $\pi_1 \Lambda$.
 \end{proof}

  \noindent The following characterization of non-dilated immersed walls will be used in \cite{Woodhouse214}.

  \begin{prop} \label{prop:partition}
 Let $X$ be tubular space, and $(\wt{X}, \mathcal{W})$ the wallspace obtained from a finite set of immersed walls in $X$. 
  The horizontal walls in $\mathcal{W}$ can be partitioned into a collection $\mathcal{A}$ of subsets such that:
  \begin{enumerate}
   \item \label{part:1} The partition is preserved by $G$,
   \item \label{part:2} The walls in $A$ are pairwise non-intersecting for each $A \in \mathcal{A}$,
   \item \label{part:3} For each $\wt{X}_{\wt{v}}$ only finitely many $A \in \mathcal{A}$ contain walls intersecting $\wt{X}_{\wt{v}}$
   \item \label{part:4} Let $\wt{\Lambda} \in A \in \mathcal{A}$ be a wall intersecting $\wt{X}_{\wt{v}}$. There exists $h \in G_{\wt{v}}$ perpendicular to $\wt{\Lambda} \cap \wt{X}_{\wt{v}}$ such that $A = \{ h^r \wt{\Lambda} \}_{r\in \mathbb{Z}}$.
  \end{enumerate}
 \end{prop}

   \begin{proof}
    Let $\wt{\Lambda}$ be a horizontal wall in $\mathcal{W}$.
    There are finitely many $\stab(\wt{\Lambda})$-orbits of lines of the form $\wt{\Lambda} \cap \wt{X}_{\wt{w}}$.
    Let $\{\ell_i \}^{n}_{i=1}$ be a set of representatives for these orbits and let $\{\wt{v}_i\}_{i=1}^n$ be vertices in $\wt{\Gamma}$ such that $\ell_i = \wt{\Lambda} \cap \wt{X}_{\wt{v}_i}$.
    By Proposition \ref{lemma:disjointWallOrbit}, for each $i$ there exist $h_i \in G_{\wt{v}_i}$ perpendicular to $\ell_i$ such that $h_i \wt{\Lambda} = gh_ig^{-1}\wt{\Lambda} $ for all $g \in \stab(\wt{\Lambda})$, and each $A_i = \{ h_i^r \wt{\Lambda} \}_{r\in \mathbb{Z}}$ is a collection of pairwise disjoint walls.
    Let $\mathcal{A}_i = \{ gA_i \mid g \in G\}$.
    If $A_i \cap gA_i \neq \emptyset$ then there exist $r,s$ such that $h^s \wt{\Lambda} = g h^r \wt{\Lambda}$ which implies $h^{-s} g h^r \in \stab(\wt{\Lambda})$.
    Therefore, for all $t \in \mathbb{Z}$
    \[
     h^{-s}gh^{t+r} \wt{\Lambda} = ( h^{-s} g h^r) h^t (h^{-s} g h^r)^{-1} \wt{\Lambda} = h^{\pm t}\wt{\Lambda},
    \]

    \noindent so $gh^{t + r} \wt{\Lambda} = h^{\pm t+s}\wt{\Lambda}$, and so $gA_i = A_i$.
    Therefore each $\mathcal{A}_i$ is a partition of the $G$-orbit of $\wt{\Lambda}$ satisfying (\ref{part:1}) and (\ref{part:2}).

    By Lemma \ref{lemma:WallOrbits}, there exists $p_i, q_i$ such that $h_1^{p_i} \wt{\Lambda} = h_i^{q_i}\wt{\Lambda}$.
    Let $p = p_1 \cdots p_n$, and $\hat{q}_i = p_1, \cdots p_{i-1} q_i p_{i+1} \cdots p_n$.
    Let $h = h_1^p$.
    Let $A = \{h^r \wt{\Lambda} \}_{r\in \mathbb{Z}}$ and $\mathcal{A} = \{gA \mid g \in G \}$.
    Note that $\mathcal{A}$ is a common refinement of the partitions $\mathcal{A}_i$, still satisfying (\ref{part:1}) and (\ref{part:2}) since $\{h_i^r \wt{\Lambda} \}_{r\in \mathbb{Z}} \supseteq \{ h_i^{\hat{q}_i r} \wt{\Lambda} \}_{r\in\mathbb{Z}} = \{h^r \wt{\Lambda}\}_{r\in \mathbb{Z}}$.
    Condition (\ref{part:4}) holds since if $\wt{\Lambda}$ intersects $\wt{X}_{\wt{w}}$ then there exists $i$ and $f \in \stab(\wt{\Lambda})$ such that $\wt{w} = f \wt{v}_i$ so $A = \{ f h_i^{\widehat{q}_ir} f^{-1} \wt{\Lambda} \}$.

    Suppose that $\wt{X}_{\wt{w}}$ is intersected by some infinite set $\{ g_j A\}_{j=1}^{\infty}$.
    Then there exists an $i$ such that $g_j \wt{\Lambda} \cap \wt{X}_{\wt{w}} = g_j \ell_i$ for infinitely many $j$.
    Fix one such $g_k$, then  there is an infinite set $\{g_k^{-1}g_jA\}$ such that all the walls in $g_k^{-1}g_jA$ intersect $\wt{X}_{g_k^{-1}\wt{w}}$.
    By our choice of $g_k$, there exists $f \in \stab(\wt{\Lambda})$ such that $\wt{\Lambda} \cap \wt{X}_{g_k^{-1}\wt{w}} = f\ell_i$.
    But since $A = \{ h^r\wt{\Lambda} \}_{r \in \mathbb{Z}} = \{h_i^{\widehat{q}_i r} \wt{\Lambda} \}_{r\in\mathbb{Z}} = \{fh_i^{ \widehat{q}_i r}f^{-1} \wt{\Lambda} \}_{r\in\mathbb{Z}}$, the $G_{g_k^{-1}\wt{w}}$-orbit of $f \ell_i$ is contained in finitely many $G$-translates of $A$.
    This contradicts that there is an infinite set $\{g_k^{-1}g_jA\}$ such that all the walls in $g_k^{-1}g_jA$ intersect $\wt{X}_{g_k^{-1}\wt{w}}$, so (\ref{part:3}) holds.
 \end{proof}

\section{Computing the Dilation Function}

In this section the results of Section \ref{Contracting} are applied to concrete examples to determine whether or not they are finite dimensional.
In this section we assume that all immersed walls are non-trivial, in the sense that they do not consist of a single immersed circle.
This allows us to identify $\stab(\wt{\Lambda})$ with $\pi_1 \Lambda$.
The main focus will be on computing the dilation function $R: \pi_1 \Lambda \rightarrow \mathbb{Q}^*$ of an immersed wall.
Let $q:\Lambda \rightarrow \Omega$ be the quotient map obtained by quotienting each circle in the equitable set to a vertex.
This simplifies the computation since $R$ factors through $q_{*}: \pi_1\Lambda \rightarrow \pi_1\Omega$.

\begin{displaymath}
 \xymatrix{ \pi_1 \Lambda \ar[r]^R \ar[d]_{q_{*}} & \mathbb{Q}^* \\
            \pi_1 \Omega \ar[ur]   & \\
 }
\end{displaymath}

Regard $\Omega$ as a directed graph by fixing orientations of its edges $E(\Omega)$, choosing the orientation of each edge $\sigma$ to be consistent with all other edges mapping into the same vertex space of $X$.
Define a \emph{weighting} $\omega: E(\Omega) \rightarrow \mathbb{Q}^*$ as follows: for each directed edge $\sigma \in E(\Omega)$ let $\omega(\sigma) = \frac{\#[C_{\iota}: \ola{f}_e]}{\#[C_{\tau}: \ora{f}_e]}$,
where $X_{e}$ is the edge space $\sigma$ maps into, and $C_{\iota}, C_{\tau}$ are the elements in the equitable set attached to the initial and terminal ends of $\sigma$.
Let $\gamma$ be a combinatorial path in $\Lambda$ representing an element $[\gamma] \in \pi_1\Lambda$, and let $\overline{\gamma}$ be the combinatorial path in $\Omega$ obtained by quotienting the circle-edges of $\gamma$ to vertices, then $q \circ \gamma$ is homotopic to $\overline{\gamma}$.
If $\overline{\gamma} = \sigma_1^{\epsilon_1} \cdots \sigma_n^{\epsilon_n}$ with $\epsilon \in \{\pm1\}$, then $R([\gamma]) = \omega(\sigma_1)^{\epsilon_1} \cdots \omega(\sigma_n)^{\epsilon_n}$.
Note that this does not depend on the choice of representative $\gamma$. 
See Figure \ref{fig:schematicArcChoiceDependence} for examples of this quotient and edge weightings.

\begin{exmp} \label{exmp:DirectInfiniteDimensional}
 Let $G = \la a,b,s,t \mid [a,b], s^{-1}abs = a^2, t^{-1}abt = b^2 \ra$.
 The group $G$ is a free-by-cyclic tubular group that acts freely on a CAT(0) cube complex, but cannot act on a finite dimensional CAT(0) cube complex.
 The proof uses the dilation function.

 The equitable set $\{a, b \}$ demonstrates that $G$ acts freely on a CAT(0) cube complex.
 Any collection of circles parallel to $ab^{-1}$ extends to an immersed wall $\wt{\Lambda}$, which must be non-dilated since $\wt{\Lambda}$ cannot regularly intersect its translates.
 We claim that any immersed wall containing a circle not parallel to $ab^{-1}$ is dilated.
 Such an immersed wall must exist since an equitable set for $G$ must generate a finite index subgroup of the vertex group.
 Suppose there is an immersed wall $\Lambda$ with equitable set $S = \{{v}_1, \ldots, {v}_n \}$, where ${v}_i = a^{x_i}b^{y_i}$ where we assume $v_1 \neq a^nb^{-n}$.

 The equitable set $S$ must satisfy the equations
 \[
  2 \sum_{i=1}^n | x_i | = 2 \sum_{i=1}^n | y_i | = \sum^n_{i=1} | x_i - y_i|. \label{eq:equitable} \tag{$*$}
 \]

 The claim follows by finding a closed path $\gamma$ in $\Omega$ such that $R([\gamma]) \neq \pm1$.
 Direct the edges of $\Omega$ such that edges exiting the vertex space via the attaching maps $ab$ are the initial ends.
 Therefore the number of edges leaving the vertex corresponding to ${v}_i$ is $2|x_i - y_i|$, while the number of vertices arriving at that vertex is $2|x_i| + 2|y_i|$.
 If some $|x_i | + |y_i | > |x_i - y_i |$, then since each $|x_i|+ |y_i| \geq |x_i - y_i|$ we have

 \[
  \sum_{i=1}^n |x_i | + |y_i | > \sum_{i=1}^n |x_i - y_i |,
 \]

 \noindent which would contradict \eqref{eq:equitable}.
 Therefore the number of edges entering and exiting each vertex in $\Omega$ are equal, so there exists a \emph{directed Eulerian trail} $\gamma$, which is a cycle traversing each edge in $\Omega$ precisely once.
 Therefore

 \[
  |R([\gamma])| = \prod^n_{i=1} \frac{| x_i - y_i |^{2| x_i - y_i |}}{ (2| x_i |)^{2| x_i |} (2| y_i | )^{2 | y_i |}},
 \]

 \noindent and the claim will be proven by showing that $|R([\gamma])| <1$.
 Considering each factor in the product separately we want to show that
 \[
  \frac{|x-y|^{2|x-y|}}{(2|x|)^{2|x|}(2|y|)^{2|y|}} \leq 1
 \]
 with equality only when $x= -y$.
 Since $v_1 \neq a^{n}b^{-n}$, the first term in the product is strictly less than one, so $|R([\gamma])| < 1$.

 This is immediate in the trivial case when $|x-y|=0, |x|=0,$ or $|y|=0$.
 After taking the square root and applying the logarithm, the non-trivial case is equivalent to showing that
 \[
  |x-y|\log|x-y| \leq |x|\log(2|x|) + |y|\log(2|y|), \label{eq:convex} \tag{$**$}
 \]
 \noindent with equality only when $x = -y$.
 As $z \log(z)$ is strictly convex for $z>0$, the following inequality holds with equality when $p = q$;
 \[
  \frac{|p+q|}{2}\log\Big(\frac{|p+q|}{2}\Big) \leq \frac{|p|}{2} \log{|p|} + \frac{|q|}{2} \log{|q|}.
 \]
 \noindent Thus \eqref{eq:convex} holds by letting $p = 2x$ and $q = -2y$.
\end{exmp}

\begin{exmp} \label{RAAGexample}

 The Right Angled Artin Group $A =\langle a,b,c,d \mid [a,b] = [b,c] = [c,d] = 1 \rangle$ is not subgroup separable (see \cite{NibloWiseNonEngulfing}).
 From the alternative presentation $A = \langle a_i, b_i \mid  [a_i, b_i], a_1 = a_2, b_2 = b_3 ;1 \leq i \leq 3 \rangle$ it is clear that $A$ is fundamental group of a tubular group given by three tori and two cylinders.
 Consider an immersed wall with the following equitable set: $\{a_1b_1^2, b_1\}, \{a_2b_2^{-1}, a_2b_2^2 \}$ and $\{ a_3^2 \}$, and the arcs given in Figure \ref{fig:SchematicDiagram}.
 The diagram shows the underlying graph $\Omega$ with the associated edge weightings.
 There is a single simple closed path with $R([\gamma]) = (2) \cdot (1/2) \cdot (1/2)^{-1} \cdot (1) = 2$, therefore $\wt{\Lambda}$ is dilated and $\stab(\wt{\Lambda})$ is not separable (see discussion in Section \ref{introduction}).
 Furthermore by Theorem \ref{theorem:MainA} there must exist an infinite cube in the associated cubulation.

  \begin{figure}
 \includegraphics[scale=0.15]{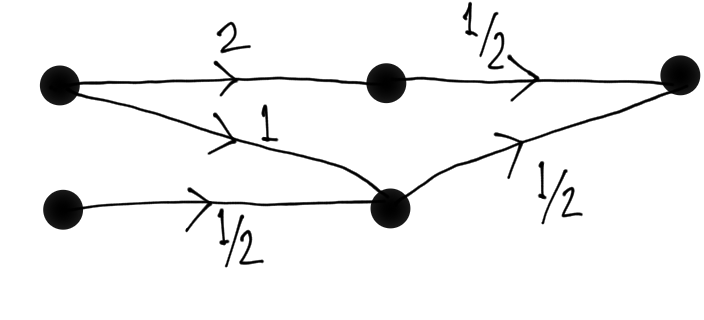}
  \caption[.]{The graph $\Omega$ from Example \ref{RAAGexample}. }
  \label{fig:SchematicDiagram}
\end{figure}
\end{exmp}

\begin{exmp} \label{arcChoiceDependent}
 The group $G = \langle a,b,t \mid [a,b], t^{-1}at = b \rangle$ from Example \ref{example1} demonstrates that the choice of arcs connecting circles can change whether or not a cubulation is finite dimensional.
 Consider the equitable set $\{a^2b, ab^2 \}$ which extends to both a dilating wall and a non-dilating wall shown in Figure \ref{fig:schematicArcChoiceDependence}.

\begin{figure}
 \includegraphics[scale=0.15]{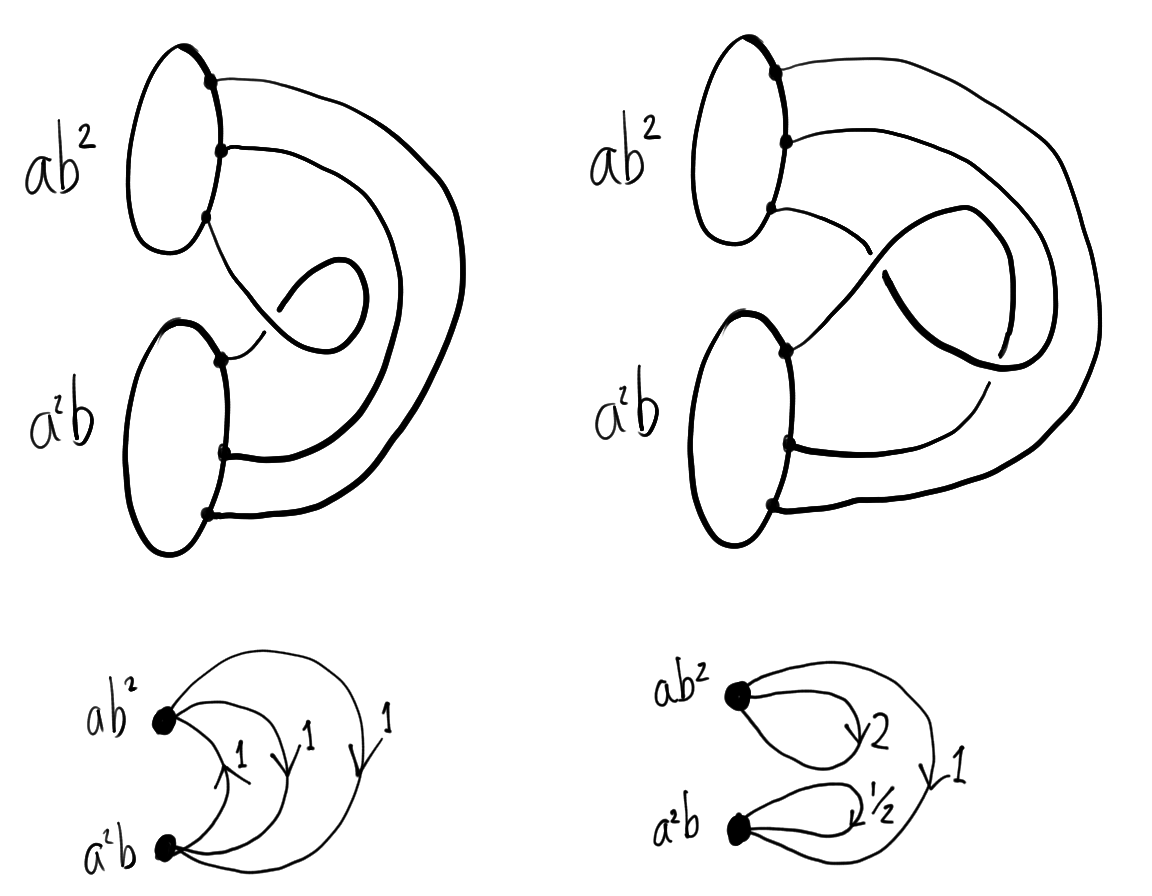}
 \caption{Example \ref{arcChoiceDependent}. The immersed walls $\Lambda_1, \Lambda_2$ are shown on the top, and the associated quotient $\Omega_i$ is shown below with its orientation and weighting. The immersed wall on the left is non-dilated and the immersed wall on the right is dilated.}
 \label{fig:schematicArcChoiceDependence}
\end{figure}

\end{exmp}

  \section{Dilated Walls are not Quasi-Isometrically Embedded} \label{QIEWalls}

  \noindent Using Theorem \ref{theorem:MainA}, Theorem \ref{mainD} can be restated in the following form.

  \begin{thm} \label{nonQIembed}
   If $\wt{\Lambda}$ is dilated then $\stab(\wt{\Lambda}) \leqslant \pi_1 X$ is not quasi-isometrically embedded.
  \end{thm}

  \noindent The following example motivates the proof of Theorem \ref{nonQIembed}, and illustrates that a non-dilated wall can fail to be quasi-isometrically embedded.

  \begin{exmp} \label{spiralExmp}
   Consider the tubular group $G = \langle a,b,t \mid [a,b], t^{-1} a t = b \rangle$ from Example \ref{example1}.
   Let $\Lambda$ be the immersed wall given by the equitable set $\{ a^{-1}b \}$
   which intersects each end of the edge space exactly once (see Figure \ref{fig:spiralSpace}).
   Observe that $\Lambda$ is embedded, and therefore non-dilated.
   Let $H = \stab(\wt{\Lambda}) \leqslant G$.
   Note that $H$ is a rank $2$ free group, with basis $\{a^{-1}b, t\}$.
   We will show that $H$ is quadratically distorted in $G$.

   \begin{figure}
 \includegraphics[scale=0.14,keepaspectratio=true]{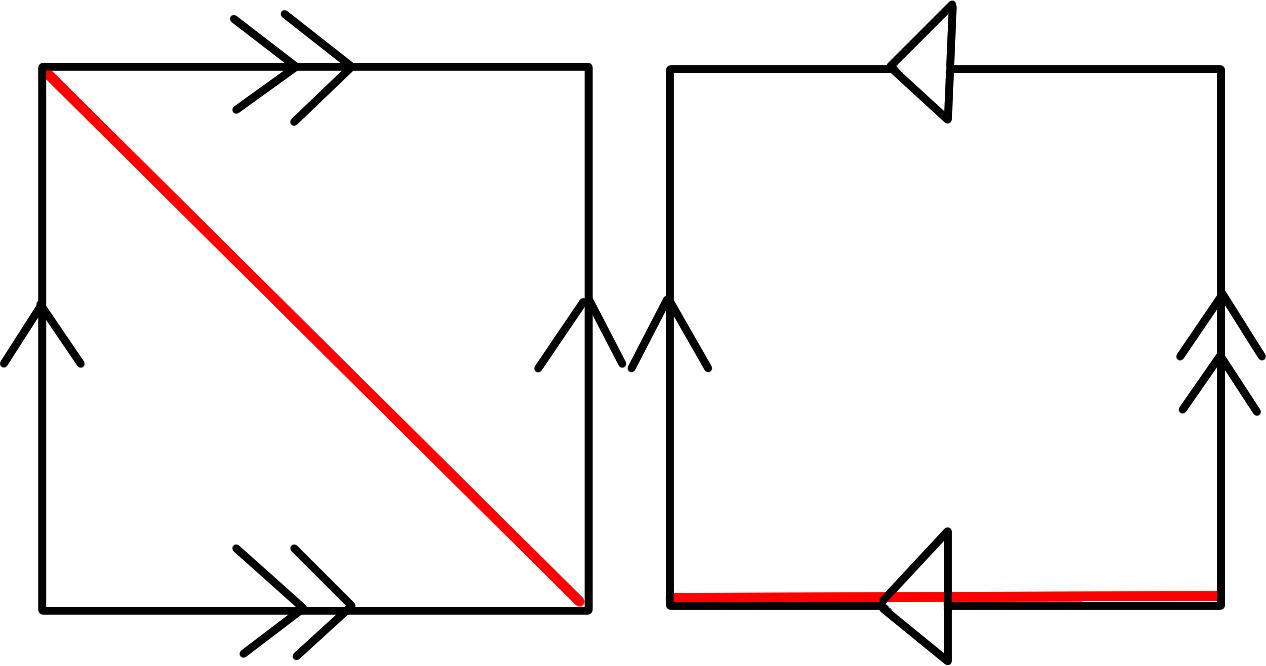}
 \caption{The tubular space and immersed wall in Example \ref{spiralExmp}}
 \label{fig:spiralSpace}
\end{figure}

    Define $\alpha_0 = 1$ and for each $n \geq 1$ define $\alpha_n = t(a^{-1}b)^n \alpha_{n-1}$. This sequence of elements has the property that $|\alpha_n |_{_{H}} = \sum^n_{i=1} (i+1) = \frac{1}{2} n(n+1) + n$.
    The element $\beta_n = a^n(ta^{-1})^n \in G$ satisfies $|\beta|_{_G} \leq 3n$.
    It can be verified by induction that $\alpha_n = \beta_n$.
    Therefore $H$ is a quadratically distorted subgroup of $G$.
    Figure \ref{fig:Spiral} illustrates the paths these elements correspond to in the universal cover: $\alpha_n$ spirals outwards increasing in length quadratically, while $\beta_n$ cuts through inside, increasing linearly in length before traveling out by another linear factor.

    \begin{figure}
 \includegraphics[scale=0.2,keepaspectratio=true]{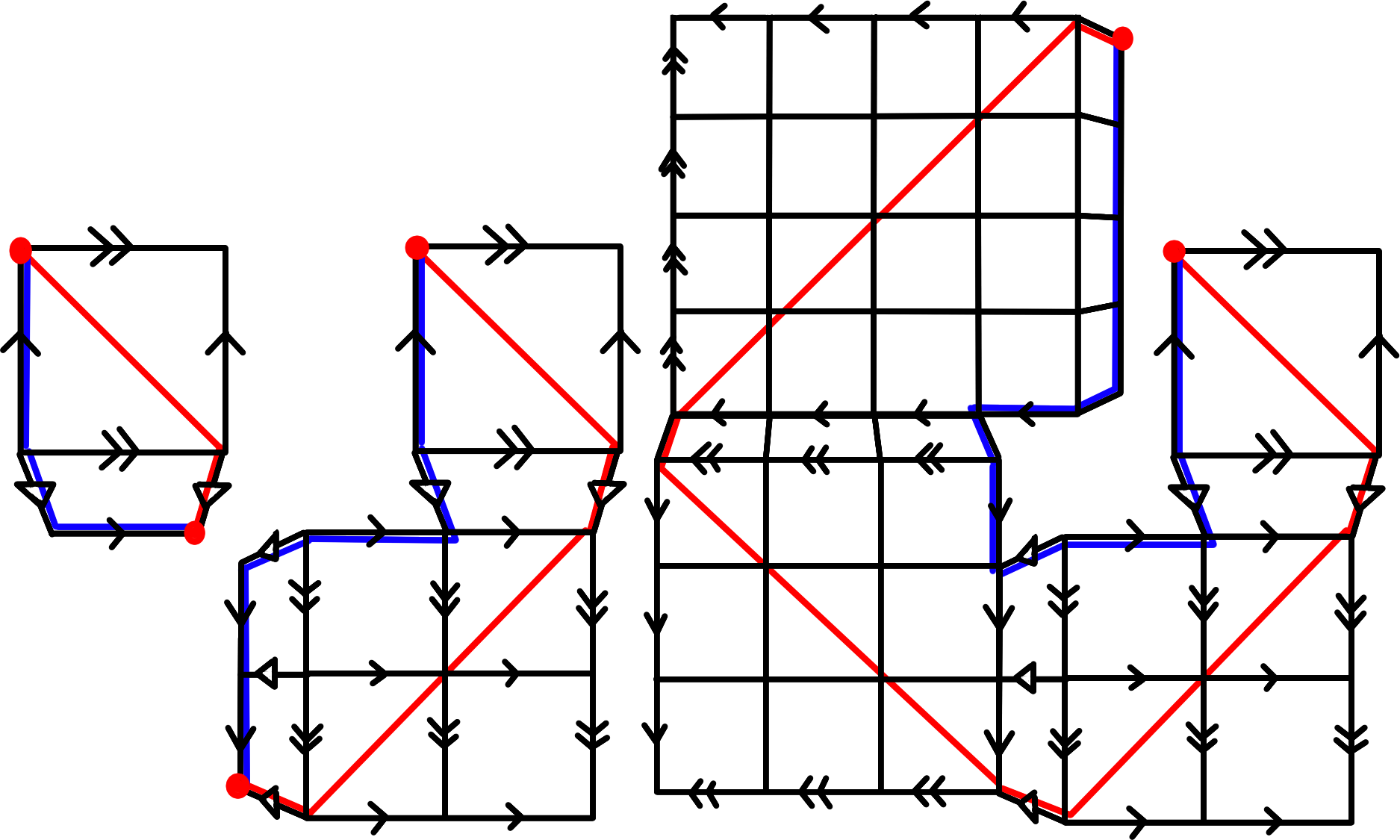}
 \caption{The paths corresponding to $\alpha_1$, $\beta_1$, and $\alpha_2$, $\beta_2$, and $\alpha_4$, $\beta_4$ }
 \label{fig:Spiral}
\end{figure}
  \end{exmp}

   Before embarking on the proof of Theorem \ref{nonQIembed} we outline the strategy employed.

  \begin{enumerate}
   \item \label{step1} Inspired by the geometric interpretation of Example \ref{spiralExmp}, construct a sequence of spiraling paths $\gamma_m$ in $\Lambda$ determined by a sequence of parameters.
   \item \label{step2} Choose the parameters of $\gamma_m$ such that the length of $\gamma_m$ in $\Lambda$ is a quadratic function of $m$.
   \item \label{step3} Double up $\gamma_m$, as illustrated in Figure \ref{fig:DoubleSpiral}, to produce a second sequence of paths $\rho_m$ that spiral out before spiraling back in, with the length growing quadratically in $\Lambda$.
   This stage is necessary because in Example \ref{spiralExmp} the rate at which $\gamma_m$ spiraled outwards was linear.
   \item \label{step4} Specify a path homotopy from $\rho_m$ to $\rho'_m$, where the length of $\rho'_m$ in $X$ is a linear function of $m$.
  \end{enumerate}

  \begin{proof}[Proof of Theorem \ref{nonQIembed}]

   Let $\Lambda$ be a dilated immersed wall in $X$. Let $z \in \pi_1 \Lambda$ satisfy $|R(z)| \neq 1$, and let $\gamma$ be a closed immersed path in $\Lambda$ representing $z$.

      \underline{\textbf{Stage \ref{step1}:}} Decompose $\gamma$ as
   \[
    \gamma = x_0 \cdot y_1 \cdot x_1 \cdot y_2 \cdots x_{\ell -1} \cdot y_{\ell}
   \]
   \noindent where each $x_i$ is a path in a circle of $\Lambda$ immersed in a vertex space $X_{u_i}$, and $y_i$ is an arc of $\Lambda$ embedded in an edge space $X_{e_i}$.
   As $\gamma$ is a closed path, the terminal point of $y_{\ell}$ is the initial point of $x_0$.
   The subpath $x_i$ is \emph{diagonal} if $\ora{X}_{e_{i}}$ and $\ola{X}_{e_{i+1}}$ are not parallel in $X_{u_i}$, and the subpath $y_i \cdot x_{i} \cdots x_{i+r} \cdot y_{i+r+1}$ is \emph{straight} if $x_{i}, x_{i+1}, \ldots, x_{i+r}$ are not diagonal.
   The subscripts are considered modulo $\ell$ so that subpaths of $\gamma$ containing the initial point are considered.
   Thus, after cyclically parameterizing, $\gamma$ decomposes as
   \[
    \gamma = a_0 \cdot b_1 \cdot a_1 \cdot b_2 \cdots a_{n-1} \cdot b_{n}
   \]
   \noindent where each $a_i$ is diagonal and each $b_i$ is straight.
   For notational purposes later these subscripts are considered modulo $n$.
   \begin{claim}
   There is at least one diagonal subpath in $\gamma$.
   \end{claim}
   \begin{proof}
   Let $\wt{\Lambda}$ be a wall in $\wt{X}$ stabilized by $\pi_1 \Lambda$, and $\wt{X}_{\wt{u}_0}$be a vertex space intersected by $\wt{\Lambda}$.
   Let $\wt{u}_1, \ldots, \wt{u}_{\ell-1}$ be the vertices in $\wt{\Gamma}$ on the geodesic between $\wt{u}_0$ and $\wt{u}_{\ell} = z\wt{u}_0$.
   Let $\wt{e}_i$ be the edge between $\wt{u}_{i-1}$ and $\wt{u}_i$.
   Let $\la g_i \ra = G_{\wt{e}_i}$ for $1 \leq i \leq \ell$, let $g_0 \in G_{\wt{u}_i}$ be perpendicular to $\wt{\Lambda} \cap \wt{X}_{\wt{u}_0}$ and $g_{\ell+1} = zg_0z^{-1}$.
   Let $\la \rho_i \ra = \stab(\wt{\Lambda}) \cap G_{\wt{u}_i}$ for $0 \leq i \leq \ell$.
   Suppose that the claim is false.
   This would imply that exists a primitive $h_i \in G_{\wt{u}_i}$ and $p_i, q_i \in \mathbb{Z} - \{0\}$ such that $g_{i} = h_i^{p_i}$ and $g_{i+1} = h_i^{q_i}$ for $1 \leq i \leq \ell-1$ and $g_{\ell} = zh_0^{p_{\ell}}z^{-1}$ and $g_1 = h_0^{q_{\ell}}$.
   Since $G$ acts freely on a CAT(0) cube complex it cannot contain a subgroup isomorphic to $\la r,t \mid t^{-1} r^n t r^m \ra$ where $n \neq \pm m$ \cite{HaglundSemiSimple}.
   Therefore $\frac{p_1 \cdots p_{\ell}}{q_1 \cdots q_{\ell}} = \pm 1$, so Lemma \ref{lemma:WallOrbits} says
   \[
    R(z) =  \prod_{i=0}^{\ell} \frac{\#[\rho_i, g_{i+1} ]}{\#[\rho_i, g_{i} ]} =
    \frac{\#[\rho_0,h_0^{q_{\ell}}]}{\#[\rho_0, g_0]} \frac{\#[\rho_{\ell}, zg_0z^{-1}] }{\#[\rho_{\ell}, zh_0^{p_{\ell}}z^{-1}]} \prod_{i=1}^{\ell-1} \frac{\#[\rho_i, h_i^{q_i} ]}{\#[\rho_i, h_i^{p_i} ]}
    = \frac{q_1 \cdots q_{\ell}}{p_1 \cdots p_{\ell}}
    = \pm 1,
   \]
   \noindent since $\#[\rho_0,h_0] = \#[\rho_{\ell}, zh_0z^{-1}]$ and $\#[\rho_0, g_0] = \#[\rho_{\ell}, zg_0z^{-1}]$. This contradicts our choice of $\gamma$.
   \end{proof}

   Parameterise the circle in $\Lambda$ containing $a_i$ as an immersed path $c_i$ with $c_i(0) = c_i(1) = a_i(1)$ and such that if $a_i(0) \neq a_i(1)$ there is an immersed path $a_i'$ satisfying $c_i^{k_i} = (a_i')^{-1} \cdot a_i$ for some $k_{i} \geq 1$.
   If $a_i(0) = a_i(1)$ then let $c_i^{k_i} = a_i$, and let $a_i' = a_i^{-1}$.
   Note that $a_i(0) = a_i'(0)$ and $a_i(1) = a_i'(1)$.
   Define a sequence of immersed paths $\{\gamma_m : I \rightarrow \Lambda \}$ inductively by setting $\gamma_{-1} = a_0(0)$, and 

   \[
    \gamma_m = \gamma_{m-1} \cdot \prod_{i=0}^{n-1} \hat{a}_{mn +i} \cdot c_i^{s(mn +i)} \cdot b_{i+1} 
   \]
   \noindent where $\hat{a}_{mn +i} \in \{ a_i, a_i' \}$ and $s(mn + i)\in \mathbb{Z} - \{ 0 \}$.
    Both $\hat{a}_{mn+i}$ and $s(mn+i)$ will be defined inductively to replicate the spiralling effect in Example~\ref{spiralExmp}.
   Note that for $\gamma_m$ to be an immersed path, if $\hat{a}_{mn+i} = a_i$ then $s(mn+i) > 0$, and if $\hat{a}_{mn+i} = a_i'$ then $s(mn+i) < 0$.

     \begin{figure}
 \centering
 \includegraphics[scale=0.3]{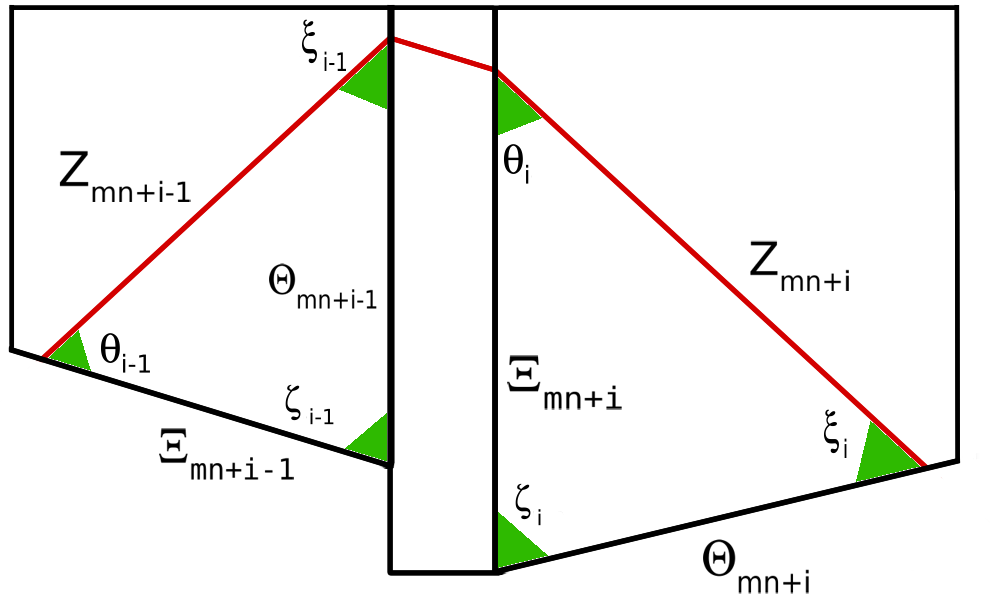}
 \caption{A portion 
 of $\gamma_m$ traversing $X_{v_{i-1}}$, $Y_i$, and $X_{v_i}$ is illustrated schematically with the corresponding side lengths labeled.}
 \label{fig:twoTriangles}
\end{figure}

   \underline{\textbf{Stage \ref{step2}:}}
   Let ${X}_{{v}_i}$ be the vertex space containing the image of $a_i$.
   Observe that for each $b_i$ there is a cylinder $Y_i \looparrowright X$, that $b_i$ factors through, where $Y_i$ is a union of cylinders that cover edge spaces and cylinders that immerse in the vertex spaces.
   Let $\ola{Y}_{i}$ and $\ora{Y}_{i}$ denote the boundary components of $Y_i$ with
   $\ola{f_i} : \ola{Y}_i \rightarrow X_{v_{i-1}}$ and $\ora{f_i} : \ora{Y}_i \rightarrow X_{v_{i}}$ denoting the corresponding restrictions of the immersion.
   When $\gamma$ traverses every vertex space diagonally then each $Y_i$ is an edge space.

   We choose $\hat{a}_{mn+i} \in \{a_i , a_i'\}$ inductively by first setting $\hat{a}_1 = a_1$.
   Assuming that $\hat{a}_{mn+i-1}$ has been chosen we specify that $s(mn+i-1) \in \integers - \{0\}$ is positive or negative so that $\hat{a}_{mn+i-1}  \cdot c_i^{s(mn+i-1)}$ is a locally geodesic path in $X_{v_i}$.
    Choose $\hat{a}_{mn+i}\in \{a_i , a_i'\}$ such that the terminal point of $\hat{a}_{mn+i-1} \cdot c_i^{s(mn+i-1)}$ meets $\ola{Y}_i$ with an acute angle on the same side as the initial point of $\hat{a}_{mn+i} \cdot c_i^{s(mn+i)}$ leaves $\ora{Y}_i$.
   This ensures that $\gamma_m$ spirals in a consistent direction.

   Inside the torus $X_{v_i}$ the paths $\ora{f_{i}}$, $\ola{f_{i+1}}$ and $c_i$ determine a triangle with angles $\theta_i$ between the $\ora{f_{i}}$ and $c_i$ sides, $\xi_i$ between the $\ola{f_{i+1}}$ and $c_i$ sides and $\zeta_i$ between the $\ora{f_{i}}$ and $\ola{f_{i+1}}$ sides.
   The rule of sines states
   \[
    \frac{Z_i}{\sin \zeta_i} = \frac{\Xi_i}{\sin \xi_i} = \frac{\Theta_i}{\sin \theta_i}
   \]
   \noindent where the value of the numerator is the length of the side opposite the angle in the denominator (see Figure \ref{fig:twoTriangles}). Let $Z_{mn+i} = |\hat{a}_{mn+i} \cdot c_i^{s(mn +i)}|$.
   The subscripts of $\zeta_i, \xi_i$, and $\theta_i$ are considered modulo $n$, while the subscripts of $Z_i, \Xi_i$, and $\Theta_i$ are not.
   Note that we chose $\hat{a}_{mn+i}$ and the sign of $s(mn+i)$ so that $0 < \xi_i, \theta_i < \pi / 2$.

    Let $r_i =
    \frac{1}{|\ola{f_i}|} \frac{\sin \xi_i}{\sin \zeta_i}$, and $t_i = \frac{1}{|\ora{f_i}|} \frac{\sin \theta_i}{\sin \zeta_i}$ with $r_{i+n}=r_i$ and $t_{i+n} = t_i$.
   Define $s(mn+i) \in \integers - \{0\}$ inductively by setting $s(0) = 1$, and assuming that $s(mn+i-1)$ is defined choose $|s(mn+i)|$ to be large enough that
     \begin{align*}
     1  & \leq \; r_iZ_{mn+i} - t_{i-1}Z_{mn+i-1} \label{eq:lowerbound} \tag{$\dagger$} \\
      & = r_i \:( | \hat{a}_{mn+i} \cdot c_i^{s(mn+i)} | )  - \; t_{i-1} (|\hat{a}_{mn+i-1} \cdot c_{i-1}^{s(mn + i -1)}| )
    \end{align*}

   \noindent and small enough that

     \[
     r_i \:( | \hat{a}_{mn+i} \cdot c_i^{s(mn+i)} | )  - \; t_{i-1} (|\hat{a}_{mn+i-1} \cdot c_{i-1}^{s(mn + i -1)}| )
     \leq 1 + r_i | c_i^{1 + k_i}| . \label{eq:upperbound} \tag{$\ddagger$}
    \]

     \noindent Combining (\ref{eq:lowerbound}) and (\ref{eq:upperbound}), replacing $| \hat{a}_{mn+i} \cdot c_i^{s(mn+i)} |$ with $Z_{mn+i}$, and applying the rule of sines produces the following inequality:

     \[
    1 \leq \;
    \frac{ \Xi_{mn+ i}}{|\ola{f_i}|} - \frac{ \Theta_{mn + i-1}}{|\ora{f_i}|}
    \leq 1 + r_i |c_i^{1+k_i}| . \label{eq:nastypasty} \tag{$\star$}
   \]

   \noindent Geometrically, this means that $s(mn+i)$ is defined so that consecutive triangles in the sequence have adjacent sides increasing at a bounded rate (see Figure~\ref{fig:twoTriangles}).
   The $|\ola{f}_i|$ and $|\ora{f}_i|$ accounts for the different lengths of the attaching maps of $Y_i$.
   To estimate a lower bound on the length of $\gamma_m$, first rewrite (\ref{eq:lowerbound}) as

   \[
    Z_{mn+i} \geq r_i^{-1} \big[ 1+ t_{i-1} Z_{mn+i-1} \big].
   \]

   \noindent Without loss of generality $\prod_{i=0}^{n-1} r_i^{-1}t_i = \prod_{i=0}^{n-1}  \frac{\sin \theta_i}{\sin \xi_i} \frac{|\ola{f}_i|}{|\ora{f}_i|} \geq 1$, otherwise replace $\gamma$ with $\gamma^{-1}$, which switches $\theta_i$ for $\xi_i$ and $\ola{f}_i$ for $\ora{f}_i$. Applying this formula recursively produces
   \[
    Z_{mn+i} \geq r_i^{-1} \big[1 + t_{i-1}r_{i-1}^{-1}\Big[1 + t_{i-2}r_{i-2}^{-1} \bigg[ \cdots  \Bigg[ 1+ t_{i-n} Z_{mn+i-n} \Bigg] \cdots \bigg] \Big] \big],
   \]
   \noindent hence
   \[
    Z_{mn+i} \geq \Bigg[ \prod^{n-1}_{i=0} r_i^{-1}t_i \Bigg]  Z_{m(n-1)+i} + \cdots \; \geq \; Z_{m(n-1)+i} + D,
   \]
   \noindent where $D > 0$ is a lower bound on the lower order terms, which are independent of any value of $Z_{mn+i}$, and the value of $i$. Therefore there is a lower bound on the length of $\gamma_m$:
   \[
    |\gamma_m| \geq \sum^{m(n+1) -1}_{j=0} Z_j \geq \sum_{j=0}^m jD \geq \frac{D}{2} m (m+1).
   \]
   \noindent Hence the length of $\gamma_m$ grows at least quadratically.


   \underline{\textbf{Stage \ref{step3}:}} Define $\rho_m := \gamma_m \cdot a_0 \cdot c_0 \cdot a_0^{-1} \cdot \gamma_m^{-1}$, an immersed path with length growing quadratically with $m$.
   See Figure \ref{fig:DoubleSpiral}.

  \begin{figure}
 \centering
 \includegraphics[scale=0.2]{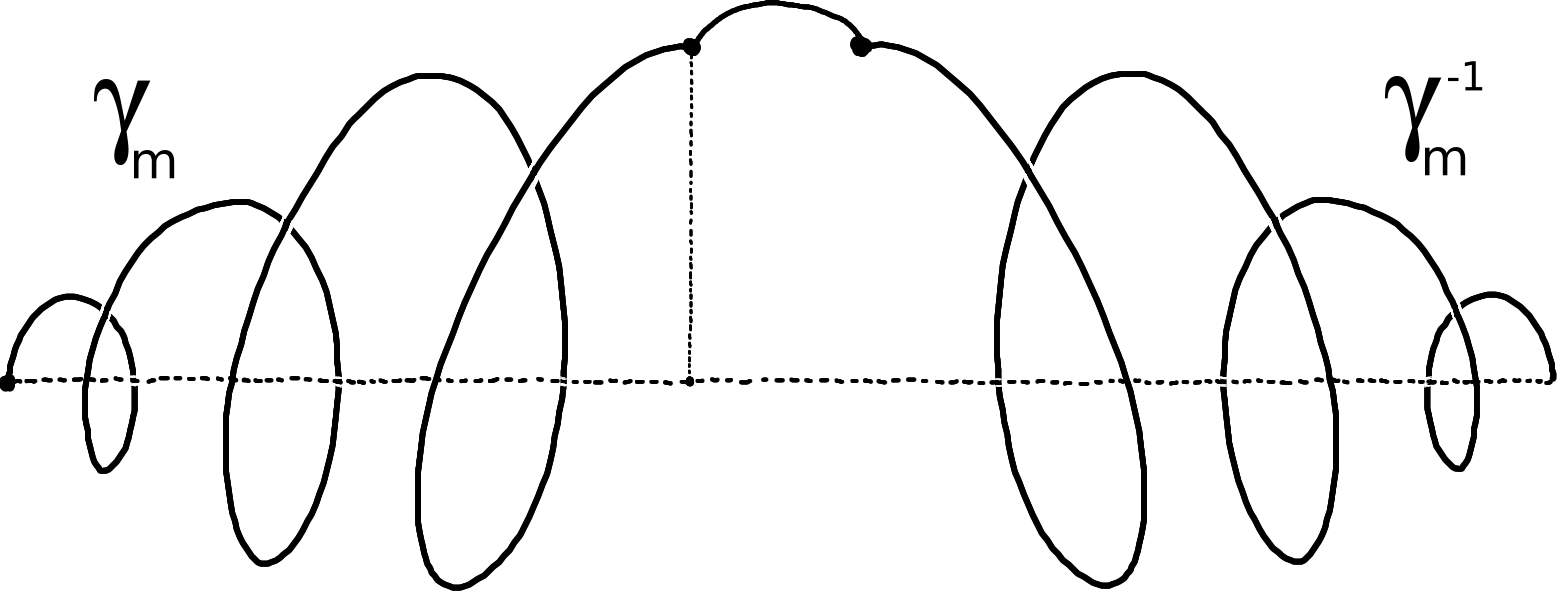}
 \caption{The spiral $\gamma_m$ is doubled up.}
 \label{fig:DoubleSpiral}
\end{figure}

   \underline{\textbf{Stage \ref{step4}:}} A path homotopy of $\rho_m$ in $X$ will be described in two stages, and the resulting curve will have length a linear function of $m$.
   Compare the following description to Figure \ref{fig:homotope}.

   \emph{Step 1:} Path homotope each diagonal segment $\hat{a}_{mn+i} \cdot c_i^{s(mn+i)}$ in $\gamma_m$ across the vertex space $X_{v_i}$ to a path $p_{mn+i} \cdot q_{mn+i}$ where $p_{mn+i}$ is a geodesic with image in $X_{v_i} \cap Y_i$ and $q_{mn+i}$ is a geodesic with image in $X_{v_i} \cap Y_{i+1}$ (see the top right diagram in Figure \ref{fig:homotope}).
   These paths respectively have lengths $\Xi_{mn+i}$ and $\Theta_{mn+i}$.
   The same path homotopies are made to $\gamma_m^{-1}$.

   \emph{Step 2:} Path homotope each $q_{mn+i-1} \cdot b_{i} \cdot p_{mn+i}$ factoring through $Y_{i}$, and the corresponding inverse path, to a local geodesic factoring through $Y_{i}$ joining the initial and terminal points (see the bottom right diagram in Figure \ref{fig:homotope}).
   From the inequality \eqref{eq:nastypasty}, we can deduce that the resulting path has length at most $|b_{i-1}| + |\ora{f}_i|\big(1 + r_i |c_i^{1 + k_i}|\big)$.
   The segment $q_{mn-1} \cdot b_n \cdot a_0 \cdot c_0 \cdot a_0^{-1} \cdot b_n^{-1} \cdot q_{mn-1}^{-1}$ can similarly be homotoped since $b_n \cdot a_0 \cdot c_0 \cdot a_0^{-1} \cdot b_n^{-1}$ factors through a cylinder $Y \looparrowright X$.
   \begin{figure}
 \centering
 \includegraphics[scale=0.3]{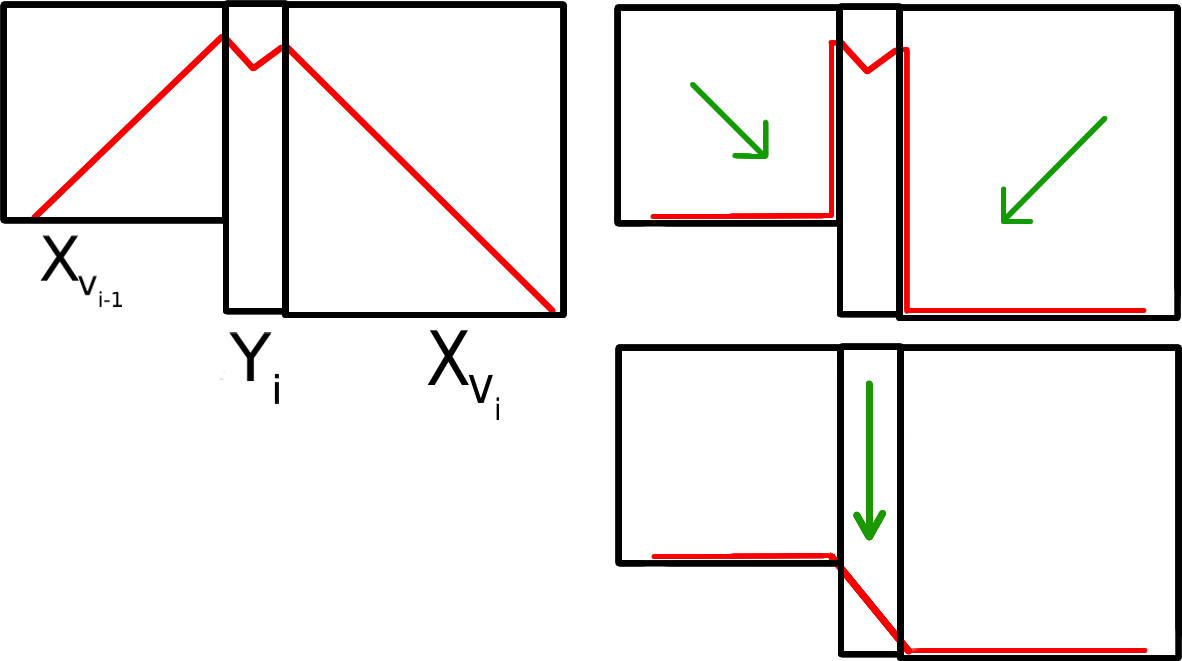}
 \caption{The left diagram shows $X_{v_{i-1}}$, $Y_i$, and $X_{v_i}$ containing $a_{i-1}\cdot c_{i-1}^{s(mn+i-1)}\cdot b_i \cdot a_{i}\cdot c_{i}^{s(mn+i)}$. The top right diagram shows the first step of the homotopy, and the bottom right diagram shows the second step.}
 \label{fig:homotope}
\end{figure}
   Therefore the length of the resulting curve $\rho_m'$ is less than
   \[
    2mn \max_{0\leq i < n} \Big\{ |b_{i-1}| + |\ora{f}_i|\big(1 + r_i |c_i^{1 + k_i}| \big) \Big\} + |a_0 \cdot c_0 \cdot a_0^{-1}|
   \]
   which is a linear function of $m$.
   Thus $\wt{\Lambda}$ is distorted.
 \end{proof}

\bibliographystyle{plain}
\bibliography{Ref}
\nocite{*}
\end{document}